\documentclass[a4paper,11pt]{article}
\usepackage{setspace} 

\usepackage{ragged2e}

\usepackage{titlesec}
\usepackage{cfr-lm}
\usepackage{stmaryrd}
\usepackage{cite}
\usepackage[all]{xy}
\usepackage[dvips]{graphicx}
\usepackage{graphpap}
\usepackage{ifthen}
\usepackage{enumerate}
\usepackage{amsmath}
\usepackage{mathrsfs}
\usepackage[errorshow]{tracefnt}
\usepackage{fancybox}
\usepackage{amsfonts}
\usepackage{amsmath}
\usepackage{amssymb}
\usepackage{amsthm}
\usepackage{multirow}
\usepackage{array}
\usepackage{mathtools}
\usepackage{soul}
\usepackage{pict2e}
\DeclareMathAlphabet{\mathpzc}{OT1}{pzc}{m}{it}

\DeclareFontFamily{U}{matha}{\hyphenchar\font45}
\DeclareFontShape{U}{matha}{m}{n}{
      <5> <6> <7> <8> <9> <10> gen * matha
      <10.95> matha10 <12> <14.4> <17.28> <20.74> <24.88> matha12
      }{}
\DeclareSymbolFont{matha}{U}{matha}{m}{n}

\DeclareMathSymbol{\pm}            {2}{matha}{"08}
\DeclareMathSymbol{\mp}            {2}{matha}{"09}
\DeclareMathSymbol{\varleftarrow}{3}{matha}{"D0}
\DeclareMathSymbol{\varrightarrow}{3}{matha}{"D1}
\DeclareMathSymbol{\vee}           {2}{matha}{"5F}
\DeclareMathSymbol{\wedge}         {2}{matha}{"5E}
\DeclareMathSymbol{\leq}         {3}{matha}{"A4}
\DeclareMathSymbol{\geq}         {3}{matha}{"A5}
\DeclareMathSymbol{\in}            {3}{matha}{"50}
\DeclareMathSymbol{\owns}          {3}{matha}{"51}

\usepackage{bbold}

\usepackage{tikz}

\makeatletter
\DeclareRobustCommand{\Lcorner}{\mathbin{\mspace{1mu}\text{\L@corner}\mspace{1mu}}}
\newcommand{\L@corner}{%
  \setlength{\unitlength}{\fontcharht\font`T}%
  \begin{picture}(0.8,0)
  \roundcap
  \Line(0.1,0.95)(0.1,0.05)
  \Line(0.1,0.05)(0.7,0.05)
  \end{picture}%
}
\makeatother

\makeatletter
\DeclareRobustCommand{\Tri}{\mathbin{\mspace{1mu}\text{\L@corneer}\mspace{1mu}}}
\newcommand{\L@corneer}{%
  \setlength{\unitlength}{\fontcharht\font`T}%
  \begin{picture}(0.8,0)
  \roundcap
  \Line(0.1,0.052)(1.0,0.052)
  \end{picture}%
}
\makeatother

\makeatletter
\DeclareRobustCommand{\Irt}{\mathbin{\mspace{1mu}\text{\L@corneeer}\mspace{1mu}}}
\newcommand{\L@corneeer}{%
  \setlength{\unitlength}{\fontcharht\font`T}%
  \begin{picture}(0.8,0)
  \roundcap
  \Line(0.1,0.82)(1.0,0.82)
  \end{picture}%
}
\makeatother

\newcommand{\Loc}{{\scriptscriptstyle\mkern 1mu\Lcorner}}
\newcommand{\LLoc}{{\scriptscriptstyle\mkern 1mu\mathbb{L}}}
\newcommand{\loc}{{\scriptscriptstyle\mkern 1mu\ell}}
\newcommand{\tri}{{\wedge\scriptscriptstyle\mkern -13.9mu\Tri}}
\newcommand{\irt}{{\!\vee\scriptscriptstyle\mkern -13.9mu\Irt}}

\newcommand{\glob}{\text{\scalebox{1.2005}{\rotatebox[origin=c]{180}{$\loc$}}}}

\makeatletter
\newcommand{\thickhline}{%
    \noalign {\ifnum 0=`}\fi \hrule height 1pt
    \futurelet \reserved@a \@xhline
}
\newcolumntype{'}{@{\hskip\tabcolsep\vrule width 1pt\hskip\tabcolsep}}
\makeatother
\newcolumntype{"}{@{\hskip\tabcolsep\vrule width 1.5pt\hskip\tabcolsep}}
\makeatother

\newcommand{\scr}{\mathscr}
\newcommand{\scrG}{\mathscr{G}}
\newcommand{\scrB}{\mathscr{B}}

\newcommand{\mangd}{M}

\newcommand{\crt}{{\!\vee\!}}

\def\boxit#1{\vbox{\hrule\hbox{\vrule\kern3pt
             \vbox{\kern3pt#1\kern3pt}\kern3pt\vrule}\hrule}}

\newcommand{\beq}{\begin{equation}}
\newcommand{\beqn}{\begin{equation*}}
\newcommand{\eeq}{\end{equation}}
\newcommand{\eeqn}{\end{equation*}}
\newcommand{\beqa}{\begin{eqnarray}}
\newcommand{\beqan}{\begin{eqnarray*}}
\newcommand{\eeqa}{\end{eqnarray}}
\newcommand{\eeqan}{\end{eqnarray*}}
\newcommand{\bdm}{\begin{displaymath}}
\newcommand{\edm}{\end{displaymath}}

\newcommand{\ba}{\begin{array}}
\newcommand{\ea}{\end{array}}

\newcommand\nn{\nonumber}

\newcommand\benu{\begin{enumerate}}
\newcommand\eenu{\end{enumerate}}
\newcommand\bit{\begin{itemize}}
\newcommand\eit{\end{itemize}}

\theoremstyle{plain}
\newtheorem{theorem}{Theorem}[section]
\newtheorem{lemma}[theorem]{Lemma}
\newtheorem{cor}[theorem]{Corollary}
\newtheorem{prop}[theorem]{Proposition}

\theoremstyle{definition}
\newtheorem{defi}[theorem]{Definition}
\newtheorem{ex}[theorem]{Example}
\newtheorem{rmk}[theorem]{Remark}

\def\der'{\mathfrak{der}'\,}
\def\der{\mathfrak{der}\,}
\def\str'{\mathfrak{str}'\,}
\def\str{\mathfrak{str}\,}

\def\sl{\mathfrak{sl}}
\def\gl{\mathfrak{gl}}

\def\qed{\hspace{\stretch{1}} $\Box$
\noindent}

\usetikzlibrary{positioning}

\newcommand{\de}{\delta}

\newcommand{\dlb}{\llbracket}
\newcommand{\drb}{\rrbracket}

\newcommand{\blb}
{\text{$\llbracket$\hspace{-4pt}\scalebox{0.99}{$|$}\hspace{-2.58pt}\scalebox{0.99}{$|$}\hspace{-2.58pt}\scalebox{0.99}{$|$}}}
\newcommand{\brb}
{\text{$\rrbracket$\hspace{-4pt}\scalebox{0.99}{$|$}\hspace{-2.58pt}\scalebox{0.99}{$|$}\hspace{-2.58pt}\scalebox{0.99}{$|$}}}

\newcommand{\ad}{\mathrm{ad}\,}

\def\fg{{\mathfrak g}}

\def\sh{\sharp}

\def\*{\partial}

\numberwithin{equation}{section}

\RequirePackage{hyperref} 

\begin{document}

\frenchspacing

\titleformat{\section}
{\normalfont\bfseries\large\sffamily}{\textsf{{\thesection}}}{1em}{}
\titlelabel{\textsf{{\thetitle}}\quad}

\titleformat{\subsection}
{\normalfont\bfseries\sffamily}{\textsf{{\thesubsection}}}{1em}{}
\titlelabel{\textsf{{\thetitle}}\quad}

\titleformat{\subsubsection}
{\normalfont\bfseries\sffamily}{\textsf{{\thesubsubsection}}}{1em}{}
\titlelabel{\textsf{{\thetitle}}\quad}

\includegraphics[height=2cm]{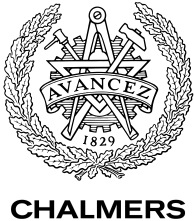}
\hspace{2mm}
\includegraphics[height=1.85cm]{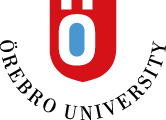}

\vskip-10pt
\hfill {\tt \today} \\
\vskip-10pt

\vspace*{1cm}

\vspace*{1.5cm}

\begin{center}
\noindent
{\LARGE {\sf \textbf{Cartanification of contragredient Lie superalgebras}}}\\
\vspace{.3cm}

\renewcommand{\thefootnote}{\fnsymbol{footnote}}

\vskip 1truecm

\noindent
{\large {\sf \textbf{Martin Cederwall${}^{\,\mathsf 1}$ and Jakob Palmkvist${}^{\,\mathsf 2}$}
}}\\
\vskip .5truecm
        ${}^{\mathsf 1\,}${\it{Department of Physics\\ Chalmers University of Technology\\ SE-412 96 Göteborg, Sweden}\\[3mm]}
        {\tt martin.cederwall@chalmers.se}\\
\noindent
\vskip .5truecm
        ${}^{\mathsf 2\,}${\it{School of Science and Technology\\\"Orebro University\\ SE-701 82 \"Orebro, Sweden}\\[3mm]}
        {\tt jakob.palmkvist@oru.se} \\
\end{center}

\vskip 1cm

\centerline{\sf \textbf{
Abstract}}
\vskip .2cm

Let $\scrB$ be a $\mathbb{Z}$-graded Lie superalgebra equipped with an invariant $\mathbb{Z}_2$-symmetric homogeneous bilinear form
and containing a grading element. Its local part (in the terminology of Kac) $\scrB_{-1} \oplus \scrB_{0}\oplus\scrB_{1}$ 
gives rise to another $\mathbb{Z}$-graded Lie superalgebra,
recently constructed in \cite{Cederwall:2022oyb}, that we here denote
$\scrB^W$ and call the cartanification of $\scr B^W$, since it is of Cartan type in the cases where it happens to finite-dimensional.
In cases where $\scr B$ is given by a generalised Cartan matrix, we compare $\scrB^W$ to the tensor hierarchy algebra $W$
constructed from the same generalised Cartan matrix by a modification of the generators and relations. 
We generalise this construction
and give conditions under which $W$ and $\scrB^W$ are isomorphic, proving a conjecture in \cite{Cederwall:2022oyb}. 
We expect
that the algebras with restricted associativity underlying the cartanifications
will be useful in applications of tensor hierarchy algebras to the field of extended geometry in physics.

\noindent

\newpage

\pagestyle{plain}

\tableofcontents

\section{Introduction} \label{intro}

Since they were first introduced a decade ago, \cite{Palmkvist:2013vya}
tensor hierarchy algebras have appeared in various guises
and applications to physics, describing
structures known as tensor hierarchies in 
gauged supergravity theories and extended geometry
\cite{deWit:2008ta,Greitz:2013pua,Bossard:2017wxl,Cagnacci:2018buk,Bossard:2019ksx,Cederwall:2019bai,Bonezzi:2019bek,Bossard:2021ebg,Borsten:2021ljb,
Cederwall:2021xqi,Cederwall:2023xbj}.
However, there is still a lack of a sufficiently simple and comprehensive definition.

Any tensor hierarchy algebra is 
characterised by a pair $(\fg,\lambda)$,
where $\fg$ is a Kac--Moody algebra and $\lambda$ is a dominant integral weight of $\fg$.
It is a Lie superalgebra with a
consistent $\mathbb{Z}$-grading such that the subalgebra at degree $0$ includes the derived algebra $\fg'$ of $\fg$,
and the $\fg'$-module at degree $1$ includes the lowest weight $\fg'$-module
with lowest weight $-\lambda$, denoted by $R(-\lambda)$.
For such pairs $(\fg,\lambda)$, two different types of tensor hierarchy algebras have been defined,
$S=S(\fg,\lambda)$ and $W=W(\fg,\lambda)$.
For $\fg=A_{n-1}=\mathfrak{sl}(n)$ and $\lambda=\Lambda_1$
(the fundamental weight associated to one of the outermost nodes in the Dynkin diagram),
they are both finite-dimensional and, being also simple,
appear in the
classification of such
Lie superalgebras
as $S(n)$ and $W(n)$, which are Lie superalgebras of \textit{Cartan type} \cite{Kac77B}.
Algebras are here considered over some field $\mathbb{K}$ which is algebraically closed and of characteristic zero.

To the pair $(\fg,\lambda)$ we can also associate a contragredient Lie superalgebra $\scrB$,
related to the tensor hierarchy algebras $S$ and $W$ (and a forerunner to the latter in some applications \cite{Palmkvist:2015dea,Cederwall:2018aab}, based on ideas in \cite{Cremmer:1998px,Henry-Labordere:2002xau,Henneaux:2010ys}).
It is a Borcherds--Kac--Moody algebra with a consistent $\mathbb{Z}$-grading $\scrB=\bigoplus_{k\in\mathbb{Z}}\scrB_k$,
where the subalgebra $\scrB_0$ contains $\fg'$ along with a one-dimensional subalgebra,
and $\scrB_1=R(-\lambda)$. Furthermore, $\scrB_2=\big(R(-\lambda)\vee R(-\lambda)\big)\ominus R(-2\lambda)$,
where $\vee$ denotes symmetrised tensor product.
Since $\scrB$ is contragredient (unlike the tensor hierarchy algebras $S$ and $W$),
the $\fg$-modules $\scrB_k$ and $\scrB_{-k}$ are dual for any $k\in\mathbb{Z}$.

Currently, there are essentially three definitions or constructions of
the tensor hierarchy algebras 
$S$ (and corresponding ones for $W$):

\begin{enumerate}
\item[(i)] 
{\it The original $\fg$-covariant construction for finite-dimensional $\fg$
{\rm\cite{Palmkvist:2013vya}}.} Here, always 
$S_1=R(-\lambda)=\scr B_1$ and $S_0=\fg\subset \scr B_0$. Then $S_{-1}$ is the maximal 
nontrivial subspace of ${\rm Hom}(S_1,S_0)$ 
such that for any $\Theta \in S_{-1}$, the kernel of the map
$S_1 \vee S_1 \to S_1$ that $\Theta$
induces by $\Theta(x \vee y)=[\Theta(x),y]-[x,\Theta(y)]$
includes the $\fg$-module $R(-2\lambda)$. 
This construction is obviously limited to cases where there is such a nontrivial subspace of ${\rm Hom}(S_1,S_0)$.
It is related to constructions of similar structures in \cite{Lavau:2017tvi,Lavau:2019oja,Lavau:2020pwa,LavauPalmkvist}.

\item[(ii)]
{\it The $\gl(r)$-covariant construction for $(\fg,\lambda)=(E_r,\Lambda_1)$ {\rm\cite{Bossard:2017wxl,Bossard:2019ksx,Bossard:2021ebg}}.}
Here (and anywhere below where $E_r$ is considered), $\Lambda_1$ is the fundamental weight associated to node $1$ in Figure \ref{ediagram}
(otherwise, we use the Bourbaki numbering of nodes.)
This construction is based on a
non-consistent $\mathbb{Z}$-grading, associated to node $r$,
with the Cartan-type Lie superalgebra
$W(r)$ at degree $0$. When $\fg$ is infinite-dimensional, $r\geq9$, 
extra modules
appear as quotients besides
$\fg'$ and $R(-\lambda)$ in $S_0$ and $S_1$.
\item[(iii)]
{\it The construction by generators and relations {\rm\cite{Carbone:2018xqq,Cederwall:2019qnw,Cederwall:2021ymp}}.} This
is a modification of the Chevalley--Serre presentation of $\scr B$ from
its Cartan matrix or Dynkin diagram. Also here, the structure of extra module appears.
This definition is very general, but for infinite-dimensional $\fg$ it is often difficult to determine whether the 
relations generate a proper ideal of the free algebra on the generators.
The explicit form of all Serre relations is not known.
\end{enumerate}

\begin{figure}
\begin{center}
\begin{picture}(220,60)
(110,-5)
\put(113,-10){${\scriptstyle{1}}$}
\put(153,-10){${\scriptstyle{2}}$}
\put(202,-10){${\scriptstyle{r-4}}$}
\put(242,-10){${\scriptstyle{r-3}}$}
\put(282,-10){${\scriptstyle{r-2}}$}
\put(322,-10){${\scriptstyle{r-1}}$}
\put(260,45){${\scriptstyle{r}}$}
\thicklines
\multiput(210,10)(40,0){4}{\circle{10}}
\multiput(215,10)(40,0){3}{\line(1,0){30}}
\put(155,10){\circle{10}}
\put(115,10){\circle{10}}
\put(120,10){\line(1,0){30}}
\multiput(160,10)(35,0){2}{\line(1,0){10}}
\multiput(175,10)(10,0){2}{\line(1,0){5}}
\put(250,50){\circle{10}} \put(250,15){\line(0,1){30}}
\end{picture}\\
\vspace{0.5cm}
\footnotesize{\caption{\textit{Dynkin diagram of $E_r$ with our labelling of nodes.}\label{ediagram}}}
\end{center}
\end{figure}
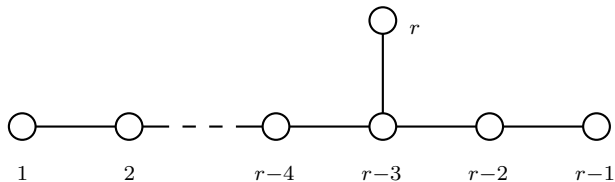

In addition, tensor hierarchy algebras constructed by generators and relations
were in \cite{Cederwall:2022oyb} shown to be homomorphic to Lie superalgebras constructed in a fourth way
under certain conditions, including that $\fg$ be finite-dimensional.
In the present paper, we call this fourth construction (or the result of it)
{\it cartanification}, since it gives the Lie superalgebra of Cartan type $W(n)$
when it is applied to the basic Lie superalgebra $A(0,n-1)=\mathfrak{sl}(n|1)$.
In its present form, and in the case of infinite-dimensional $\fg$, it is unable to reproduce the extra modules 
that play a crucial role in the applications to extended geometry.
Nevertheless, we believe that we can learn much from a closer study of it.
The initial data in the cartanification consists not only of $\fg$ and $\lambda$, but also of a choice $\kappa$
of an invariant symmetric bilinear form on $\fg$.
The Cartan matrix in the construction (iii) also depends on $\kappa$, but the resulting algebra $S$ or $W$ do not.

The main aim of the present paper is to compare the definition (iii) of $W$ to the cartanification,
and to investigate to what extent the two constructions agree.
More particularly, we will investigate
under what conditions the homomorphism in \cite{Cederwall:2022oyb} is an isomorphism.
Our main result, Theorem~\ref{maintheorem},
says that this is indeed the case when:
\begin{itemize}
\item $\fg$ is simple and finite-dimensional,
\item $\lambda$ is a fundamental weight $\Lambda_i$ such that the corresponding Coxeter label is equal to one (a 
\textit{pseudo-minuscule} weight), and
\item $\kappa$ is given by $\kappa(e_i,f_i)=1$ for the generators $e_i,f_i$ corresponding to $\Lambda_i$.
\end{itemize}
However, we start with a much more general setting (where in particular $\fg$ is not necessarily simple, but semisimple)
in order to investigate to what extent structures in the tensor hierarchy algebras found previously under certain conditions  
persist also when these conditions are relaxed. The 
results are important not only in comparison with the cartanification but also in their own right.  

A central concept in the paper is that of a \textit{local algebra} in the terminology of Kac
(which has nothing to do with the concept of local rings
in ring theory). Local Lie algebras were introduced by Kac in \cite{Kac68} as the structure that remains 
when a $\mathbb{Z}$-graded Lie algebra is restricted to its \textit{local part},
which consists of the subspaces at degrees $0$ and $\pm1$.
This means that the bracket is defined, respects the $\mathbb{Z}$-grading and satisfies the Jacobi identity as long as it takes
values at these degrees, but it is not defined for pairs of elements both at degree $\pm1$
(and thus a local Lie algebra is in fact not an algebra). 
The concept was generalised to Lie superalgebras in
\cite{Kac77B}, and can obviously be generalised to any type of algebras.
It seems to not have been used much since it was introduced, but we have found it very useful to work at
the level of local algebras as long as possible and then take the step to ``global'' extensions.
Any local Lie superalgebra
can be extended uniquely (up to isomorphism) to a
$\mathbb{Z}$-graded Lie superalgebra 
such that the subalgebras at positive and negative degrees are freely generated by 
the subspaces at degree $1$ and $-1$, respectively.
This is the \textit{maximal} $\mathbb{Z}$-graded Lie superalgebra with the given local part,
and factoring out its maximal graded ideal intersecting the local part trivially
yields the \textit{minimal} one \cite{Kac68,Kac77B}.

The paper is organised as follows and outlined in Figure~\ref{organisation}. In Section \ref{sec2} we review how $\scrB$ and $W$ 
(in Sections~\ref{Contragredient
Lie superalgebras} and \ref{tha-subsection}, respectively)
are defined from the same data $(\fg,\lambda,\kappa)$
and generalise the construction (iii) above. The structure of the subspace $W_{-1}$ at degree $-1$ is studied in
Section~\ref{tha-structure}.
In Section~\ref{cartanification-section} we shift focus and define the concepts of contragredient local
Lie superalgebras and of cartanification. This is done without reference to the contents of the preceding sections,
so Sections~\ref{Local algebras} and \ref{cartansubsec} can be read independently from them. The construction is an
improved version of the one in \cite{Cederwall:2022oyb}, which is simpler (without being less general or less rigorous).
The two parts are then tied together
in Section~\ref{sec3.3} where we focus on the case where the contragredient local
Lie superalgebra is given by the local part of $\scrB$, and even more in Section~\ref{isomorphism-section},
where we give the homomorphism which we show is an isomorphism under the assumptions above.
We conclude with some open questions and directions for further work in Section~\ref{sec5}.
\\\\
\noindent
\underline{\it Acknowledgments:} JP would like to thank Jens Fjelstad, Jonas Hartwig and Dimitry Leites for
discussions and correspondence.

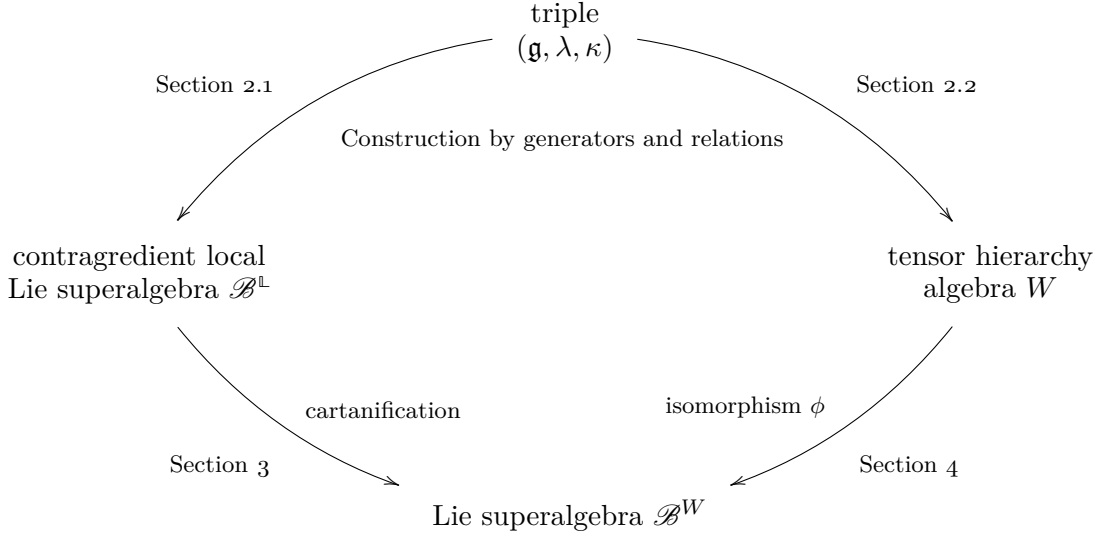
\begin{figure}[h]
\begin{align*}
\xymatrix@!0@C=5.6cm@R=3.2cm
{
&*+++\txt{triple\\$(\fg,\lambda,\kappa)$} 
*---\ar@/_2pc/[dl]_(.6)*+++{\txt{\footnotesize Section~\ref{Contragredient
Lie superalgebras}}}^*+++rrr{\txt{\footnotesize \qquad\qquad\quad Construction by generators and relations}}
\ar@/^2pc/[dr]^(.6)*+++{\txt{\footnotesize Section~\ref{tha-subsection}}} &\\  
*+++\txt{contragredient local\\Lie superalgebra
$\scr B^\LLoc$}*---\ar@/_2pc/[dr]^*+++{\txt{\footnotesize cartanification}}_(.4)*+++{\txt{\footnotesize
Section~\ref{cartanification-section}}} 
&&*+++\txt{tensor hierarchy\\algebra $W$}*---\ar@/^2pc/[dl]_*+++{\txt{\footnotesize isomorphism $\phi$}}^(0.4)*+++{\txt{\footnotesize 
Section~\ref{isomorphism-section}}}\\  
&*+++\txt{ 
Lie superalgebra $\scr B^W$}&}
\end{align*}
\caption{\textit{Organisation of the paper.}\label{organisation}}
\end{figure}

\section{From $(\fg,\lambda,\kappa)$ to $\scr B$ and $W$} \label{sec2}
We begin by introducing some notation. Let $r$ be a positive integer and set $I=\{1,\ldots,r\}$. We introduce
\begin{itemize}
\item $A=(A_{ij})_{i,j\in I}:$ a (generalised) Cartan matrix of rank $r$,
symmetrisable
and invertible, but not necessarily indecomposable;
\item $\mathfrak{g}=\fg(A):$ a Kac--Moody algebra with Cartan matrix $A$ and generators $\{e_i,f_i,h_i\}$ for ${i\in I}$ (not necessarily simple,
but semisimple);
\item $\mathfrak{h}:$ the Cartan subalgebra of $\fg$, spanned by the Cartan elements $h_i$ for $i\in I$;
\item $\kappa:$ an invariant symmetric bilinear form on $\fg$, uniquely given by $\kappa(e_i,f_j)=\delta_{ij}\epsilon_j$
for $i,j\in I$ and nonzero numbers $\epsilon_j$ such that $D^{-1}A$, where $D=\mathrm{diag}(\epsilon_i)_{i\in I}$, is a symmetric matrix;
\item $\varphi:$ the vector space isomorphism $\mathfrak{h} \to \mathfrak{h}^\ast$
given by $\varphi(h_i)(h_j)=\kappa(h_i,h_j)$;
\item $(-,-):$ the symmetric bilinear form on $\mathfrak{h}^\ast$ given by $(\mu,\nu)=\big(\varphi^{-1}(\mu),\varphi^{-1}(\nu)\big)$;
\item $\lambda=\sum_{i=1}^r \lambda_i \Lambda_i:$ a dominant integral weight of $\fg$ (the Dynkin labels $\lambda_i$ are non-negative integers).
\end{itemize}

By a weight of a Kac--Moody algebra we simply mean an element in the dual space of its Cartan subalgebra, not necessarily with
reference to a representation. For any weight $\mu$ of $\fg$, we set $h_{\mu^\crt}=\varphi^{-1}(\mu)$. For any root $\alpha$ of $\fg$,
we furthermore set $\alpha^\crt=\kappa(e_\alpha,e_{-\alpha})\alpha$, where $e_{\pm\alpha}$ are some fixed
corresponding root vectors, and $h_\alpha=h_{(\alpha^\crt)^\crt}$.
Then
\begin{align}
[e_\alpha,e_{-\alpha}]&=\kappa(e_\alpha,e_{-\alpha})\varphi^{-1}(\alpha)=h_\alpha\,, 
&[h_\alpha,e_\beta]&=\beta(h_\alpha)e_\beta=(\alpha^\crt,\beta)e_\beta\,
\end{align}
for any roots $\alpha,\beta$. Note that $h_{-\alpha}=-h_\alpha$. For any weight $\mu=\sum_{i=1}^r \mu_i \Lambda_i$ we also set
\begin{align}
\mu^\tri&=\sum_{i=1}^r \mu^\tri{}_i \Lambda_i= \sum_{i=1}^r \frac1{\epsilon_i}\mu_i \Lambda_i\,.
\end{align}
The fundamental weights $\Lambda_i$ are defined by $(\alpha_i{}^\crt,\Lambda_j)=\delta_{ij}$.

When $A$ is symmetric already before multiplying with $D^{-1}$, then there is a natural choice of $\kappa$, where 
$\kappa(e_i,f_i)=\epsilon_i=1$ for all $i\in I$. We will denote this
$\kappa$ by $\kappa_0$.

\subsection{The contragredient Lie superalgebra $\scr B$} \label{Contragredient Lie superalgebras}

We will now extend the Cartan matrix $A$ with one additional row and column $0$ to a matrix $B=(B_{ij})$, where ${i,j\in\{0\}\cup I}$, given by
\begin{align} \label{matrisdef}
B_{00}&=0\,,& B_{i0}&=-\lambda_i \,,& B_{0j}&=-\lambda^\tri{}_j\,,
& B_{ij}&=A_{ij}\,,
\end{align}
for $i,j\in I$. Thus, besides the entry $B_{00}=0$, we find $-\lambda$ in column $0$ and $-\lambda^\tri$ in row $0$.
It follows that $B$ is symmetrisable, and we also assume that $B$ is invertible.

To the matrix $B$, we associate a {contragredient Lie superalgebra} $\scr B$ 
defined 
from a set of $3(r+1)$ generators $\mangd_{\scr B}=\{e_i,f_i,h_i\}$ 
for $i \in\{0\}\cup I$, where
$e_0$ and $f_0$ are odd, and $h_0$ is even (like, as before, $e_i,f_i,h_i$ for $i\in I$). Let $\tilde{\scr B}$ be the $\mathbb{Z}$-graded Lie superalgebra
generated by this set $M_{\scr B}$ modulo the relations 
\begin{align}
[h_i,e_j]&=B_{ij}e_j\,, & [h_i,f_j]&=-B_{ij}f_j\,, & [e_i,f_j]&=\delta_{ij}h_j  
\end{align}
with the
$\mathbb{Z}$-grading where $e_i$ and $f_i$ have degree $1$ and $-1$, respectively, for any $i \in\{0\}\cup I$. 
Then $\scr B=\tilde{\scr B}/{\scr D}$, where ${\scr D}$
is the maximal graded ideal of $\tilde{\scr B}$ intersecting the local part of $\tilde{\scr B}$ trivially
\cite{Kac77B}.
Since $B$ here satisfies the conditions of a Cartan matrix of a {\it Borcherds--Kac--Moody algebra}, 
a generalisation \cite{Ray98} of the Gabber--Kac theorem \cite{Gabber-Kac,Kac} holds, which in this case says that the ideal ${\scr D}$
is generated by the Serre relations
\begin{align}
({\rm ad}\,e_i)^{1-B_{ij}}(e_j)=({\rm ad}\,f_i)^{1-B_{ij}}(f_j)=0\,.
\label{serre0}
\end{align} 
We extend $\kappa$ to an invariant bilinear form $\langle-|-\rangle$ on $\scr B$ with $\mathbb{Z}_2$-graded symmetry, uniquely given  
by $\langle e_0 | f_0\rangle=-\langle f_0 | e_0\rangle=1$ and $\langle e_i | f_i\rangle=\kappa(e_i , f_i)$ for $i\in I$.
In the same way as $\kappa$ for $\fg$,
it induces a non-degenerate symmetric bilinear form $(-,-)$ on the dual space of the Cartan subalgebra of $\scr B$ (spanned by $h_i$ for $i\in \{0\}\cup I$).
The roots of $\scr B$ are defined accordingly, and for any simple root $\alpha_i$ (where $i\in \{0\}\cup I$)
we set $\alpha_i{}^\crt=\langle e_i|f_i\rangle\alpha_i$. Thus $(\alpha_0,\alpha_0)=0$ and $\alpha_0{}^\crt=\alpha_0$.
We also note that $(\alpha_0{}^\crt,\mu)=-(\lambda,\mu)$ for any weight $\mu$ of $\fg$.
The Lie superalgebra $\scr B$ can be visualised by a Dynkin diagram where a grey node ($\otimes$), node $0$, is connected to node $i$
in the Dynkin diagram of $\fg$ with $\lambda_i$ lines. 

In the following, we will consider $\scr B$ as a $\mathbb{Z}$-graded Lie superalgebra {\it not} with respect to the ``democratic'' $\mathbb{Z}$-grading 
above, but with respect to the {\it consistent} one, where $\scr B_{\pm1}$ are odd subspaces, and the even subalgebra $\scr B_0$ is the direct sum  
of $\fg$ and a one-dimensional subalgebra.

\subsection{The tensor hierarchy algebra $W$} \label{tha-subsection}

We divide 
the index set $I$ into
a disjoint union of two subsets
\begin{align}
J&=\{\,j\in I \ |\  \lambda_j \neq 0\, \}\,,& K&=\{\,k\in I \ |\  \lambda_k = 0 \,\}\,.
\end{align}
In the Dynkin diagram of $\scr B$, they correspond to the sets of nodes connected to and disconnected from, respectively, node $0$.

We will now associate a different Lie superalgebra $W$, called tensor hierarchy algebra, to the same data $(\fg,\lambda,\kappa)$, 
following---and generalising---the definition in \cite{Carbone:2018xqq}.
In \cite{Carbone:2018xqq}, the tensor hierarchy algebra $W$ was defined for $J$ consisting of only one value,
which was chosen to always to be $1$, and moreover, the Dynkin label with this index was also equal to $1$, so that $\lambda_{i}=\delta_{1i}$.
We thus generalise this construction from fundamental weights to arbitrary dominant integral weights $\lambda$.
We also relax the condition in \cite{Carbone:2018xqq} that $\fg$ be simple and simply laced.

The construction of $W$ given in \cite{Carbone:2018xqq} starts with the 
Cartan matrix $B$ of the Lie superalgebra $\scr B$ associated to the triple $(\fg,\lambda,\kappa)$ as described in the preceding section.
The set of generators $M_{\scr B}=\{e_i,f_i,h_i\}$ for $i\in \{0\} \cup I$ of $\scr B$ is then modified
to a set $\mangd_{W}$ by replacing the odd generator $f_0$ with $|K|+1$ odd generators $f_{0i}$, where $i\in \{0\} \cup K$.
From this set of generators $\mangd_{W}$, and the Cartan matrix $B$, an auxiliary Lie superalgebra $\tilde W$
is first constructed as 
the one generated by $M_{W}$ modulo the relations
\begin{equation}
[e_i,f_j]=\delta_{ij}h_j\,,\qquad   
[h_i,f_j]=-B_{ij}f_j\,, \qquad 
[h_i,e_j]=B_{ij}e_j\,, \label{ordchevrel}
\end{equation}
\begin{equation}
[e_0,f_{0i}]=h_i\,,\qquad [h_i,f_{0j}]=-B_{i0}f_{0j}\,,\label{eifjf0a}
\end{equation}
\begin{equation}
({\rm ad}\,e_i)^{1-B_{ij}}(e_j)=({\rm ad}\,f_i)^{1-B_{ij}}(f_j)=0\,,
\label{serre1}
\end{equation}
for any $i,j,k\in \{0\}\cup I$ such that the involved generators are defined
(which means that $i \in \{0\}\cup K$ whenever $f_{0i}$ appears, and $i\in I$
whenever $f_i$ appears), and the additional relations
\begin{align}
[e_k,[f_l,f_{0i}]]= \delta_{kl}B_{il}f_{0k}&\quad\text{ for any $k,l\in K$ and $i\in\{0\}\cup K$}\,,\label{addrel1} \\
({\rm ad}\,f_j)^{1-B_{j0}}(f_{0i})=[e_j,f_{0i}]=0&\quad\text{ for any $j\in J$ and $i\in\{0\}\cup K$}\,,\label{addrel2}\\
[e_k,[e_k,f_{0i}]]=[f_k,[f_k,f_{0i}]]=0 &\quad\text{ for any $k\in K$ and $i\in\{0\}\cup K$}\,,\label{addrel4}\\
[e_j,[e_k,f_{0i}]]=0 &\quad\text{ for any $j\in J$, $k \in K$ and $i\in\{0\}\cup K$}\,.\label{addrel3}
\end{align}
Then $W$ is defined as the Lie superalgebra obtained from $\tilde W$ by factoring out the maximal ideal intersecting the local part trivially,
with respect to the consistent $\mathbb{Z}$-grading. By modifying the set of generators further
to $\mangd{}_S=\mangd{}_W\backslash\{h_0,f_{00}\}$ a Lie superalgebra $S$ (called tensor hierarchy algebra as well)
can be defined in the same way, with the relations
involving $h_0$ and $f_{00}$ removed.

\begin{rmk}
A remark on the relations (\ref{addrel1})--(\ref{addrel3}):
When $\fg$ is simply laced, only the relations (\ref{addrel1}) and (\ref{addrel2})
are needed.
It was shown in
\cite{Cederwall:2019qnw} how the relations 
(\ref{addrel4})
follow from the former in the simply laced case.
Furthermore, the relations 
(\ref{addrel3}) for $i=k$
follow from $[e_j,f_{0i}]=0$ trivially when $B_{ij}=0$ and by
\begin{align}
2[e_j,e_i,f_{0i}]&=[e_j,e_i,e_i,f_i,f_{0i}]\nn\\
&=2[e_i,e_j,e_i,f_i,f_{0i}]-[e_i,e_i,e_j,f_i,f_{0i}]\nn\\
&=4[e_i,e_j,f_{0i}]-[e_i,e_i,f_i,e_j,f_{0i}]=0\, \label{harledn}
\end{align}
when $B_{ij}=-1$ (using a notation for nested brackets where, for example, $[e_i,e_j,f_{0i}]$ means $[e_i,[e_j,f_{0i}]]$).
Then (\ref{addrel1}) can be used to show that (\ref{addrel3}) hold also for $i\neq k$.
We do not know if the relations (\ref{addrel4}) and (\ref{addrel3}) are redundant also 
when $\fg$ is not simply laced, but since we have not been able to derive them 
from (\ref{ordchevrel})--(\ref{serre1}) in the general case, we have included them here.

On the other hand, additional relations quadratic in $f_{0i}$ were included in \cite{Cederwall:2019qnw,Cederwall:2021ymp}. We conjecture that 
the corresponding ideal is contained in the maximal ideal of
$\tilde W$ intersecting the local part trivially, and thus vanishes in $W$. Although we have no general proof, we have excluded these relations
here since we will not need them and since we aim at keeping the definition as simple as possible.
\qed
\end{rmk}

\noindent
The following proposition was proven in \cite{Carbone:2018xqq} (second part of Proposition 3.2) and the proof does not use that
$\fg$ is simply laced, so it holds also in our more general setting. However, it has to be formulated slightly differently, 
with the indices of the Cartan matrix interchanged,
since we do not assume the Cartan matrix to be symmetric here.
\begin{prop} \label{flindepprop}
For any $i,j,k\in \{0\}\cup K$, we have
\begin{align}
B_{ji}[e_i,f_{0k}]&=B_{ki}[e_i,f_{0j}]\,, &
B_{ji}[f_i,f_{0k}]&=B_{ki}[f_i,f_{0j}]\,. \label{flindep}
\end{align}
\end{prop}

\noindent
The proposition says that $[e_i,f_{0j}]=0$ (and $[f_i,f_{0j}]=0$ if $i\notin J$) whenever $B_{ij}=0$ and $i\neq0$.
In particular, $f_{00}$ behaves as the ``usual'' $f_0$ in $\scr B$.
It also says that any two elements $[e_i,f_{0j}]$ and $[e_i,f_{0k}]$ (or $[f_i,f_{0j}]$ and $[f_i,f_{0k}]$ if $i\notin J$) are linearly dependent.

\begin{prop} \label{tha-struktur0}
The $\fg$-module $U(\fg)(f_{00})$ in $W_{-1}$ is an irreducible $\fg$-module with highest weight $\lambda$. 
\end{prop} 

\begin{proof}  
It follows from the defining relations 
(\ref{ordchevrel})--(\ref{addrel3}),
the additional relations (\ref{flindep}) and the discussion above that
\begin{align}
[e_i,f_{00}]&=0\,, & (\ad{f_i})^{1+\lambda_j}(f_{00})&=0\,,
\end{align}
for any $i\in I$.
\end{proof}

\begin{rmk}
Note that the element $f_{00}$, and thus
$U(\fg)(f_{00})$, is absent in $S$.\qed
\end{rmk}

\noindent
Since we can identify $f_0$ in $\scr B$ with $f_{00}$ in $W$, and all the other generators 
of $\scr B$ coincide with generators in $W$, there is a homomorphism from $\scrB$ to $W$
preserving the $\mathbb{Z}$-degree. However, if $K$ is non-empty, then $W_{-1}$ is bigger than $\scr B_{-1}$
since we can form, for example, nonzero elements $[e_i,f_{0k}]$. When $\fg$ is infinite-dimensional, these additional
elements in $W_{-1}$ are known to give rise to
additional modules as quotients besides $\scr B_0$ and $\scrB_1$ in $W_0$ and $W_1$ \cite{Bossard:2017wxl,Cederwall:2019qnw,Cederwall:2021ymp}.
In the present paper, we will 
not study the
structure of these extra modules
in $W_0$ and $W_1$, but focus on the structure of $W_{-1}$ as a $\fg$-module
that has been studied previously 
for finite-dimensional $\fg$. 
We will investigate to what extent it still appears
in our more general setting---and in the cartanification of $\scrB$, which is the topic of the next section.

Let $\fg^-$ be the semisimple subalgebra of $\fg$ generated by $\{e_k,f_k,h_k\}$ for $k\in K$. Thus the Dynkin diagram of $\fg^-$
is obtained from the one of $\scr B$ by removing the grey node as well as all ordinary white nodes connected to it.
We assume that $K$ is non-empty so that $\fg^-$ is non-trivial. (Otherwise $W$ coincides with $\scrB$ under the identification $f_{00}=f_0$.)
We also assume that the Cartan matrix of $\fg^-$ is invertible. When $|J|=1$, so that $\lambda$ is a multiple of a fundamental weight $\Lambda_j$,
this condition is equivalent to the condition that $B$ is invertible, which we have already assumed.

It is
convenient to define $f_{0\mu^\crt}$ for any weight $\mu$ of $\fg^-$ as the linear combination of 
$f_{0k}$ (for $k\in K$)
which has the same coefficients as $\mu$
has in the basis $\alpha_k{}^\crt$. We then have 
$[e_0,f_{0\mu^\crt}]=h_{\mu^\crt}$ and
the relations (\ref{flindep})
can be expressed more generally as
\begin{align}
(\mu,\alpha_i)[e_i,f_{0\nu^\crt}]&=(\nu,\alpha_i)[e_i,f_{0\mu^\crt}]\,, & 
(\mu,\alpha_i)[f_i,f_{0\nu^\crt}]&=(\nu,\alpha_i)[f_i,f_{0\mu^\crt}]\,. \label{flindepgen}
\end{align}
As the next proposition shows, it is possible to generalise this further, replacing $e_i$ and $f_i$
with $e_\alpha$ for any root in the same Weyl orbit as $\alpha_i$.
Also the other relations above involving $f_{0k}$ but not $e_j$ or $f_j$ for $j\in J$ can be
generalised accordingly.

\begin{prop}
For any real root $\alpha$ of $\fg^-$
and weights $\mu,\nu$ of $\fg^-$,
we have
\begin{align}
[e_\alpha,[e_\alpha,f_{0\mu^\crt}]]&=
0\,,\label{alphaaplha}\\
[e_\alpha,[e_{-\alpha},f_{0\mu^\crt}]]&=(\mu,\alpha)f_{0\alpha^\crt}\,,\label{alphaminusaplha}\\
(\mu,\alpha)[e_\alpha,f_{0\nu^\crt}]&=(\nu,\alpha)[e_\alpha,f_{0\mu^\crt}]\,. \label{mulambda}
\end{align}
\end{prop}
\begin{proof}
When $\alpha$ is a simple root, (\ref{alphaaplha}) and (\ref{alphaminusaplha}) follow from (\ref{addrel1}) and (\ref{addrel4}), respectively,
and (\ref{mulambda}), as we just discussed, from Proposition~\ref{flindepprop}. The idea is to generalise these special cases by applying Weyl
transformations.

For $k\in K$, the
maps $\ad{e_k}$ and $\ad{f_k}$ are locally nilpotent on 
the generators of $W$, and it follows that they are 
locally nilpotent on 
any element of $W$ (see for example \cite[Lemma~3.4]{Kac}).
Thus the exponentials 
$e^{\,\ad{e_k}}$ 
and $e^{\,\ad{f_k}}$
can be defined as endomorphisms of $W$ by their power series, and in the same way as 
when they are considered as endomorphisms of $\fg^-$, it follows that they are automorphisms of the Lie superalgebra $W$.
Furthermore, the compositions
\begin{align} \label{rk}
r_k= e^{\ad{f_k}}e^{-\ad{e_k}}e^{\ad{f_k}}\,,
\end{align}
act on the Cartan subalgebra $\mathfrak{h}$ of $\fg$ by
\begin{align}
r_k(h_{\mu^\crt})=h_{\mu^\crt}-(\alpha_k,\mu)h_k
\end{align}
and induce the fundamental Weyl reflections
\begin{align}
r_k{}^\ast(\mu)=\mu-(\alpha_k{},\mu)\alpha_k{}^\crt
\end{align}
on the dual space $\mathfrak{h}^\ast$, by 
$\varphi\big(r_k(h)\big)=r_k{}^\ast \big(\varphi(h)\big)$ for any $h \in \mathfrak{h}$.
The fundamental Weyl reflections generate the Weyl group $\scr W^-$ of $\fg^-$,
which thus is isomorphic to the group of automorphisms of $W$ generated by the elements $r_k$,
with an isomorphism $\psi$ given by $\psi(r_k{}^\ast)=r_k$.
It then
follows that $h_{w(\mu)^\crt}=\psi(w)(h_{\mu^\crt})$ for any $w\in\scr W^-$.
Furthermore, $\psi(w)(e_\alpha)$ is a root vector corresponding to $w(\alpha)$ 
for any root $\alpha$ of $\fg^-$,
so we may set $e_\alpha=\psi(w)(e_k)$ if $\alpha=w(\alpha_k)$. 
Finally, it is straightforward to check that we have
$w(f_{0\mu^\crt})=f_{0\psi(w)(\mu)^\crt}$ as expected.

Since we know that the relations (\ref{mulambda}) hold when $\alpha$ is a simple root $\alpha_k$
(where $k\in K$),
we can then for any root $\alpha$ in the Weyl orbit of $\alpha_k$
apply the automorphism $\psi(w)$, where $w$ is the element in the Weyl group mapping $\alpha_k$ to $\alpha$.
For example, the relation
\begin{align}
[e_k,[f_k,f_{0\mu^\crt}]]&=(\mu,\alpha_k)f_{0k}
\end{align}
gives
\begin{align}
[e_\alpha,[e_{-\alpha},f_{0\nu^\crt}]]&=(\mu,\alpha_k)f_{0\alpha}
\end{align}
where $\nu=w(\mu)$. Since Weyl reflections preserve the inner product, we have
\begin{align}
(\mu,\alpha_k)=\big(w(\mu),w(\alpha_k)\big)=(\nu,\alpha)\,,
\end{align}
and we get
\begin{align}
[e_\alpha,[e_{-\alpha},f_{0\nu^\crt}]]&=(\nu,\alpha)f_{0\alpha^\crt}\,,
\end{align}
which is (\ref{mulambda}).
The relations (\ref{alphaaplha}) and (\ref{alphaminusaplha}) can be shown in the same way.
\end{proof}

\noindent
Now, assume that $\fg^-$ is finite-dimensional, so that any root is real. Set
\begin{align}
e_\alpha{}^\sh=\frac1{(\mu,\alpha)}[e_\alpha,f_{0\mu^\crt}]\,
\end{align}
for any root $\alpha$ of $\fg^-$,
where $\mu$ is some weight such that $(\alpha,\mu)\neq0$.
Obviously, there is always such a weight, and thanks to (\ref{mulambda}),
the choice of $\mu$ does not matter, so $e_\alpha{}^\sh$
is well defined. For the Cartan elements $h_k$ of $\fg^-$ (where $k\in K$), set $h_k{}^\sh=-f_{0k}$.
By linearity we thus get a map $\sh$ from $\fg^-$ to the module over $\fg^-$ generated by the elements $f_{0k}$.
The following proposition says that this map $\sh$ is $\fg^-$-equivariant,
and thus that the module over $\fg^-$ generated by the elements $f_{0k}$ is isomorphic to the adjoint.

\begin{prop} \label{equivariance-prop}
If $\fg^-$ is finite-dimensional, then
$[x,y^\sh]=[x,y]^\sh$ for any $x,y\in\fg^-$.
\end{prop}
\begin{proof}
When both $x$ and $y$ are Cartan elements, then both sides are zero.
When $x$ is a root vector and $y$ is a Cartan element,
\begin{align}
[e_\alpha,h_{\mu^\crt}{}^\sh]=-[e_\alpha,f_{0{\mu^\crt}}]=-(\mu,\alpha)e_\alpha{}^\sh=[e_\alpha,h_{\mu^\crt}]^\sh\,.
\end{align}
When both $x$ and $y$ are root vectors,
set $x=e_\alpha$ and $y=e_\beta$. If $\beta=\pm\alpha$, then the proposition follows from (\ref{alphaaplha}) and (\ref{alphaminusaplha}).
If $\beta\neq\pm\alpha$, then
set
\begin{align}
\mu = \beta-\frac{(\alpha,\beta)}{(\alpha,\alpha)}\alpha
\end{align}
so that $(\mu,\alpha)=0$ and
\begin{align}
(\mu,\beta)=\frac1{(\alpha,\alpha)}\big((\alpha,\alpha)(\beta,\beta)-(\alpha,\beta)^2\big)>0
\end{align}
by the Cauchy--Schwarz inequality, since $\alpha$ and $\beta$ are linearly independent.
We then get
\begin{align}
[e_\alpha,e_\beta{}^\sh]&=\frac1{(\beta,\mu)}[e_\alpha,[e_\beta,f_{0{\mu^\crt}}]]
=\frac1{(\beta,\mu)}[[e_\alpha,e_\beta],f_{0{\mu^\crt}}]\nn\\&=\frac1{(\beta+\alpha,\mu)}[[e_\alpha,e_\beta],f_{0{\mu^\crt}}]=[e_\alpha,e_\beta]^\sh\,.
\end{align}
\end{proof}

\subsection{Structure of $W_{-1}$} \label{tha-structure}

We will now describe the structure of $W_{-1}$ as a $\fg$-module under certain conditions (which imply that $\fg$ is finite-dimensional).
First we need the following two lemmas.

\begin{lemma}\label{lemma4.5}
If $(\alpha_j{}^\crt,\alpha)=0,\pm1$ for some $j\in J$ and root $\alpha$ of $\fg^-$, then $[e_j,[e_\alpha,f_{0{\mu^\crt}}]]=0$
for any weight $\mu$ of $\fg^-$.
\end{lemma}

\begin{proof} We assume $(\vartheta,\mu)\neq0$; otherwise the lemma follows from (\ref{mulambda}).
If $(\alpha_j{}^\crt,\alpha)=0,1$, then $[e_j,e_\alpha]=0$ and
\begin{align}
[e_j,[e_\alpha,f_{0{\mu^\crt}}]]=[e_\alpha,[e_j,f_{0{\mu^\crt}}]]=0\,,
\end{align}
where the last equality follows from (\ref{addrel2}). If $(\alpha_j{}^\crt,\alpha)=-1$,
then there is a $k\in K$ such that $(\alpha_j{}^\crt,\alpha_k)=-1$
and $[e_\alpha,f_{0{\Lambda_k{}^\crt}}]$ (which is proportional to $[e_\alpha,f_{0{\mu^\crt}}]$)
can be obtained from $[e_k,f_{0{\Lambda_k{}^\crt}}]$ by acting with elements that commute with $e_j$.
Since $[e_j,[e_k,f_{0{\Lambda_k{}^\crt}}]]=0$
by (\ref{addrel3}), it follows that $[e_j,[e_\alpha,f_{0{\Lambda_k{}^\crt}}]]=0$.
\end{proof}

\begin{lemma}\label{lemma2} If $v_\mu$ is a highest weight vector of a weight $\mu$ in a $\fg$-module,
then
\begin{align}
(e_i f_i{}^{\,p}) (v_\mu)= p\big((\alpha_i{}^\crt,\mu)-p+1\big)f_i{}^{\,p-1} (v_\mu) 
\end{align}
for any $i\in I$ and any positive integer $p$.
\end{lemma}
\begin{proof}
This is a standard result easily proven by induction.
\end{proof}

\begin{theorem} \label{tha-struktur}
Suppose that $\fg^-$ is finite-dimensional with a
highest root $\vartheta$ and 
that $(\alpha_i{}^\crt,\vartheta)\geq -1$ and $\lambda_i=0,1$ for any $i\in I$.
Then the $\fg$-module $U(\fg)[e_{\vartheta},f_{0\mu^\crt}]$
is irreducible with highest weight $\vartheta+\lambda$. Furthermore, $W_{-1}$ is the direct sum of all such modules
and the module $U(\fg)(f_{00})$ with highest weight $\lambda$ (which is not present in $S_{-1}$).
\end{theorem} 

\begin{proof}  
We will show that
\begin{align}
[e_i,[e_{\vartheta},f_{0\mu^\crt}]]&=0\,,&
(\ad{f_i})^{1+(\alpha_i{}^\crt,\,\vartheta-\alpha_0)}[e_{\vartheta},f_{0\mu^\crt}]&=0\,\label{embtensor}
\end{align}
for any $i\in I=J\cup K$.

We first show the left equality.
When $i\in K$, we assume that $\alpha_i \neq \vartheta$; otherwise the equality
follows from (\ref{addrel4}). Then, as in the proof of Proposition~\ref{equivariance-prop}, we can choose
\begin{align}
\mu = \vartheta -\frac{(\alpha_i,\vartheta)}{(\alpha_i,\alpha_i)}
\end{align}
so that $(\mu,\alpha_i)=0$ but $(\mu,\vartheta)>0$. Hence
\begin{align}
[e_i,[e_{\vartheta},f_{0\mu^\crt}]]=[[e_i,e_{\vartheta}],f_{0\mu^\crt}]=0\,.
\end{align}
When $i\in J$, it follows from Lemma~\ref{lemma4.5} that $[e_i,[e_\vartheta,f_{0\mu^\crt}]]=0$.

We now turn to the right equality in (\ref{embtensor}). We have
$(\ad{f_i})^{1+(\alpha_i{}^\crt,\,\vartheta)}(e_\vartheta)=0$
and since $(\alpha_i{}^\crt,\alpha_0)\leq0$, also
$(\ad{f_i})^{1+(\alpha_i{}^\crt,\,\vartheta-\alpha_0)}(e_\vartheta)=0$.
Together with $[f_i,[f_i,f_{0\mu^\crt}]]=0$ (which holds since $\lambda_i=0,1$) we get
\begin{align}
(\ad{f_i})^{2+(\alpha_i{}^\crt,\,\vartheta-\alpha_0)}[{e_\vartheta},f_{0\mu^\crt}]=0\,.
\end{align}
Acting with $e_i$ gives
\begin{align}
0&=(\ad{e_i})(\ad{f_i})^{2+(\alpha_i{}^\crt,\,\vartheta-\alpha_0)}[{e_\vartheta},f_{0\mu^\crt}]\nn\\
&=(\ad{h_i})(\ad{f_i})^{1+(\alpha_i{}^\crt,\,\vartheta-\alpha_0)}[{e_\vartheta},f_{0\mu^\crt}]\nn\\
&\quad\,+(\ad{f_i})(\ad{e_i})(\ad{f_i})^{1+(\alpha_i{}^\crt,\,\vartheta-\alpha_0)}[{e_\vartheta},f_{0\mu^\crt}]\,.
\end{align}
Since $[e_i,[e_\vartheta,f_{0\mu^\crt}]]=0$ we may now apply Lemma~\ref{lemma2} with $p=1+(\alpha_i{}^\crt,\,\vartheta-\alpha_0)$,
from which it follows that the last line is zero.
The first term on the right hand side equals
\begin{align}
-\big(2+(\alpha_i{}^\crt,\vartheta)-(\alpha_i{}^\crt,\alpha_0)\big)(\ad{f_i})^{1+(\alpha_i{}^\crt,\,\vartheta-\alpha_0)}[{e_\vartheta},f_{0\mu^\crt}]\,.
\end{align}
Since $(\alpha_i{}^\crt,\vartheta)\geq-1$ according to the assumption, and $-(\alpha_i{}^\crt,\alpha_0)\geq0$, we have
\begin{align}
2+(\alpha_i{}^\crt,\vartheta)-(\alpha_i{}^\crt,\alpha_0)\geq 1
\end{align}
and it follows that $(\ad{f_i})^{1+(\alpha_i{}^\crt,\,\vartheta-\alpha_0)}[{e_\vartheta},f_{0\mu^\crt}]=0$.

Since they are irreducible, the direct sum of all modules $U(\fg)[e_{\vartheta},f_{0\mu^\crt}]$ and $U(\fg)(f_{00})$
is direct, and since
\begin{align}
[f_\vartheta,[e_\vartheta,f_{0\mu^\crt}]]=(\mu,\vartheta)f_{0\vartheta^\crt}
\end{align}
by (\ref{alphaminusaplha}), it contains all the generators $f_{0i}$ for $i \in \{0\}\cup I$.
Furthermore, 
$U(\fg)[e_{\vartheta},f_{0\mu^\crt}]$ is equal to $U(\mathfrak{n})[e_{\vartheta},f_{0\mu^\crt}]$,
where $\mathfrak{n}$ is the subalgebra of $\fg$ spanned by $f_i$ for $i\in I$. All these generators commute with $e_0$, so that
\begin{align}
\big({e_0\,}U(\mathfrak{g})\big)[e_{\vartheta},f_{0\mu^\crt}]=\big({e_0\,}U(\mathfrak{n})\big)[e_{\vartheta},f_{0\mu^\crt}]=
U(\mathfrak{n})[{e_0},[e_{\vartheta},f_{0\mu^\crt}]]= U(\mathfrak{n})(h_{\mu^\crt})\,.
\end{align}
It follows that $W_0=\scr B_0$, and in turn that we cannot get anything more in $W_{-1}$ than this direct sum.
\end{proof}

\section{From $\scr B$ to $\scr B^W$} \label{cartanification-section}

The local part $\scrB^\LLoc$ of the contragredient Lie superalgebra $\scrB$, together with the restriction of
$\langle-|-\rangle$
to $\scrB_{-1}\times\scrB_1$
is a special case of more general structure, which we now define.
For any element $x$ in a $\mathbb{Z}$-graded Lie superalgebra or a local Lie superalgebra 
we write $x_k$ for the component of $x$ at degree $k$.

\begin{defi}
A local Lie superalgebra $\scrG^\LLoc=\scrG_{-1}\oplus\scrG_{0}\oplus\scrG_{1}$ with a bracket $\dlb -,-\drb$ and in addition equipped with a
bilinear map
\begin{align}
\mathscr G_{-1} \times \mathscr G_{1} \to \mathbb{K}\,, \qquad (x,y) \mapsto \langle x | y \rangle\,,
\end{align}
is a \textbf{contragredient local Lie superalgebra} if 
\begin{itemize}
\item it contains a \textbf{grading element}: an element $L\in\scrG_0$ such that $\dlb L, x_k \drb= kx_k$ for any $x\in\scrG^\LLoc$,
\item the bilinear map $\langle - | - \rangle$ is invariant in the sense
that $\langle \dlb x_{-1},y_0 \drb|z_1 \rangle = \langle x_{-1} | \dlb y_0,z_1\drb \rangle$
for any $x,y,z \in \scrG^\LLoc$, and
\item the bilinear map $\langle - | - \rangle$ is homogeneous with respect to the $\mathbb{Z}_2$-grading, meaning that a
pair of homogeneous elements with different $\mathbb{Z}_2$-degrees is mapped to zero.
\end{itemize}
\qed
\end{defi}

\noindent
It is convenient to also define a corresponding bilinear map
\begin{align}
\mathscr G_{1} \times \mathscr G_{-1} \to \mathbb{K}\,, \qquad (x,y) \mapsto \langle x | y \rangle = (-1)^{xy}\langle y | x \rangle\,
\end{align}
by $\mathbb{Z}_2$-graded symmetry, where $x$ and $y$ in the exponent of $(-1)$ actually mean the $\mathbb{Z}_2$-degrees of these elements.

The local part of
$\scr B$ defined 
from a triple $(\fg,\lambda,\kappa)$ as
in Section~\ref{Contragredient Lie superalgebras}
is a contragredient local Lie superalgebra 
with grading element $L=\sum_{i=0}^r(B^{-1})_{0i}h_i$ and the map $\langle-|-\rangle$ induced by $\kappa$
\cite[Proposition~3.2]{Cederwall:2022oyb}.
(Note that different choices of $\kappa$ give rise to isomorphic contragredient
Lie superalgebras $\scrB$, but their local parts, considered as contragredient local Lie superalgebras,
are really different.)
Since this is the case that we will be mainly interested in, we have chosen the term
``contragredient local Lie superalgebra'' in lack of a better one, although it might be misleading since
the concept is more general. (In fact, any
local Lie superalgebra can be made contragredient by letting $\langle - | - \rangle$ be trivial and
extending it by hand with an additional basis element $L$, if no such element is 
already present.)

\subsection{Underlying algebra with restricted associativity} \label{Local algebras}

To any contragredient local Lie superalgebra $\scrG^\LLoc=\scrG_{-1}\oplus\scrG_{0}\oplus\scrG_{1}$,
we associate an algebra $\scrG^\glob$ in the following way. 
Let $\scrG^\glob{}_{0\pm}$ be the enveloping algebras of the Lie superalgebras $\scr G_0 \oplus F(\scrG_{\pm1})$,
where $F(\scrG_{\pm1})$ is the free
Lie superalgebra generated by $\scrG_{\pm1}$ and the adjoint action of $\scrG_{0}$ on $F(\scrG_{\pm1})$ is induced by the 
adjoint action of $\scrG_{0}$ on $\scrG_{\pm1}$ in $\scrG^\LLoc$ and the Jacobi identity. Thus in particular $\scrG^\glob{}_0=U(\scrG_0)$,
and we can consider
the tensor algebras $T(\scrG_{\pm1})=U(F(\scrG_{\pm1}))$ as subalgebras of $\scrG^\glob{}_{0\pm}$. It is then easy to see that
\begin{align}
\scrG^\glob{}_{0\pm}=U(\scrG_0)T(\scrG_{\pm1})=T(\scrG_{\pm1})U(\scrG_0)\,.
\end{align}
We now unify the two $\mathbb{Z}$-graded algebras $\scrG^\glob{}_{0\pm}$ into the direct sum vector space
\begin{align}
\scrG^\glob = \scrG^\glob{}_{-}\oplus\scrG^\glob{}_0 \oplus\scrG^\glob{}_+\,.
\end{align}
So far, $\scrG^\glob{}$ is not an algebra, but a $\mathbb{Z}$-graded vector space with a product defined for any pair of elements
which are either both in $\scrG^\glob{}_{0+}$ or both in $\scrG^\glob{}_{0-}$.
We extend it to a product defined for any pair recursively by
\begin{align}
(Xx_{\pm1})(y_{\mp1}Y)=(X(x_{\pm1}y_{\mp1}))Y\, \label{prod-def}
\end{align}
for $X \in \scrG^\glob{}_{0\pm}$ and $Y\in \scrG^\glob{}_{0\mp}$,
where $x_{\pm1}y_{\mp1}$ is given by
\begin{align}
x_{-1}y_{1} &= - \dlb y_{1},x_{-1}\drb + \langle x_{-1}|y_{1}\rangle L\,,\nn\\
x_{1}y_{-1} &= \dlb y_{-1},x_{1}\drb+\langle x_{1}|y_{-1}\rangle L +\langle x_{1}|y_{-1}\rangle\,.\label{xyplusminus}
\end{align}

\begin{rmk}Compared to \cite{Cederwall:2022oyb}, we have interchanged the arguments in the bracket which is the first term on the right hand side
of both equations (\ref{xyplusminus}). This does not change anything
when the $\mathbb{Z}$-grading is consistent, so that $x$ and $y$ are odd and the bracket is symmetric, but 
in the case of an even Lie superalgebra, this choice will turn out to go better with our other conventions.\qed
\end{rmk}

The reason for denoting the bracket by $\dlb-,-\drb$ is that we reserve $[-,-]$ for the commutator
$[x,y]=xy-(-1)^{yx}yx$ with respect to the product in $\scrG^\glob$. It is
equal to the original bracket in the following cases,
\begin{align} \label{kommutatorlika}
[x_0,y_0]&=\dlb x_0,y_0 \drb \,,& [x_0,y_{\pm1}]&=\dlb x_0,y_{\pm1}\drb \,, & [x_{\pm1},y_0]&=\dlb x_{\pm1},y_0\drb\,,
\end{align}
but not when $x \in \scrG_{\pm1}$ and $y \in \scrG_{\mp1}$. In this case we instead get
$[x_1,y_{-1}]=\langle x_1|y_{-1} \rangle$,
and it is thus important to distinguish between $[-,-]$ and $\dlb-,-\drb$. In the cases where the two brackets agree, we will
use $[-,-]$ for simplicity.

We will now see that the algebra $\scrG^\glob$ is not really associative, but almost.

\begin{prop} \label{XzY-prop}
The associative law
$(Xz)Y=X(zY)$
holds in $\scrG^\glob$ for any
\begin{align}
X &\in \scrG^\glob{}_{0\pm}=U(\scrG_0)T(\scrG_{\pm1})\,, 
& Y &\in \scrG^\glob{}_{0\mp}=T(\scrG_{\mp1})U(\scrG_0)\,, & z &\in \scrG^\LLoc \,.
\end{align}
\end{prop}
\begin{proof} This is obvious when $X$ or $Y$ belongs to $\scrG^\glob{}_0=U(\scrG_0)$, so we
assume
$X=X'x_{\pm1}$ and $Y=y_{\mp1}Y'$ where 
\begin{align}
X' &\in U_m(\scrG_0) T_p(\scrG_{\pm1})\,, & Y' &\in T_q(\scrG_{\mp1})U_n(\scrG_0) 
\end{align}  
for some non-negative integers $m,n,p,q$, 
referring to the natural filtrations of $U(\scrG_0)$ and $T(\scrG_{\pm1})$.
In the case $z=z_0$, and where the local Lie superalgebra $\scrG^\LLoc$ is even, we then proceed by induction over $m+n+p+q$.
In the base case $X'=Y'=1$ we have
\begin{align}
(x_{\pm1}z_0)y_{\mp1}&=[x_{\pm1},z_0]y_{\mp1}+(z_0x_{\pm1})y_{\mp1}=[x_{\pm1},z_0]y_{\mp1}+z_0(x_{\pm1}y_{\mp1})\nn\\
&=\pm\dlb y_{\mp1},[x_{\pm1},z_0]\drb+\langle [x_{\pm1},z_0]|y_{\mp1}\rangle L +\tfrac{1\pm1}2\langle [x_{\pm1},z_0]|y_{\mp1}\rangle\nn\\
&\quad\,\pm z_0\dlb y_{\mp1},x_{\pm1}\drb+\langle x_{\pm1}|y_{\mp1}\rangle z_0 L +\tfrac{1\pm1}2\langle x_{\pm1}|y_{\mp1}\rangle z_0\nn\\
&= \pm[ \dlb y_{\mp1},x_{\pm1}\drb, z_0]\mp \dlb [ y_{\mp1},z_0], x_{\pm1}\drb+\langle x_{\pm1}|[z_0,y_{\mp1}]\rangle L 
+\tfrac{1\pm1}2\langle x_{\pm1}|[z_0,y_{\mp1}]\rangle  \nn\\
&\quad\,\pm z_0 \dlb y_{\mp1},x_{\pm1}\drb +\langle x_{\pm1}| y_{\mp1}\rangle Lz_0 +\tfrac{1\pm1}2\langle x_{\pm1}|y_{\mp1}\rangle z_0\nn\\
&= \pm\dlb [z_0,y_{\mp1}],x_{\pm1}\drb+\langle x_{\pm1}|[z_0,y_{\mp1}]\rangle L 
+\tfrac{1\pm1}2\langle x_{\pm1}|[z_0,y_{\mp1}]\rangle  \nn\\
&\quad\,\pm\dlb y_{\mp1},x_{\pm1}\drb z_0+\langle x_{\pm1}| y_{\mp1}\rangle Lz_0 +\tfrac{1\pm1}2\langle x_{\pm1}|y_{\mp1}\rangle z_0\nn\\
&=x_{\pm1}[z_0,y_{\mp1}]+(x_{\pm1}y_{\mp1})z_0\nn\\&=x_{\pm1}[z_0,y_{\mp1}]+x_{\pm1}(y_{\mp1}z_0)=x_{\pm1}(z_0y_{\mp1})\, \label{310}
\end{align}
and in the induction step
\begin{align}
(X'x_1z_0)(y_{-1}Y')&=X'\big([x_1,z_0]y_{-1}\big)Y'+X'(z_0x_1y_{-1})Y'\nn\\
&
=X'(x_1z_0y_{-1})Y'=X'\big(x_1[z_0,y_{-1}]\big)Y'+X'(x_1y_{-1}z_0)Y'\nn\\
&=(X'x_1)\big([z_0,y_{-1}]Y'\big)+(X'x_1)(y_{-1}z_0Y')=(X'x_1)(z_0y_{-1}Y')\,. \label{311}
\end{align}
When $\scr{G}^\LLoc$ has a nontrivial odd subspace (in particular when the $\mathbb{Z}$-grading is consistent, so that the odd subspace is 
$\scrG_{-1}\oplus \scrG_{1}$), the appropriate factors of $-1$ are easily inserted in (\ref{310}) and (\ref{311}) by the sign rule.

The case $z=z_{\pm1}$ then follows from the case $z=z_0$ by 
\begin{align}
(Xz_{\pm1})Y&=(Xz_{\pm1})(y_{\mp1}Y')=\big(X(z_{\pm1}y_{\mp1})\big)Y'\nn\\
&=X\big((z_{\pm1}y_{\mp1})Y'\big)=X\big(z_{\pm1}(y_{\mp1}Y')\big)=X(z_{\pm1}Y)
\end{align}
and the case $z=z_{\mp1}$ directly by
\begin{align}
(Xz_{\mp1})Y=\big((X'x_{\pm1})z_{\mp1}\big)Y=(X'x_{\pm1})(z_{\mp1}Y)=X(z_{\mp1}Y)\,.
\end{align}
\end{proof}

\begin{cor}
The associative law
$(XZ)Y=X(ZY)$
holds in $\scrG^\glob$ when $X,Z \in \scrG^\glob{}_{0\pm}$ or $Y,Z \in\scrG^\glob{}_{0\pm}$.
\end{cor}

\begin{proof}
This follows
by writing $Z$ as a (non-commutative) polynomial of elements in $\scrG^\LLoc$ 
and using Proposition~\ref{XzY-prop} to move
one element at the time from one pair of parentheses to  the other. 
\end{proof}

\begin{rmk}
\textit{A posteriori}, the outer pair of parentheses on the right hand side of (\ref{prod-def}) is redundant.
\end{rmk}

\subsection{The cartanification} \label{cartansubsec}

We define $\scrG^\loc$ as the local part of $\scrG^\glob$.
Thus $\scrG^\loc=\scrG^\loc{}_{-1}\oplus\scrG^\loc{}_{0}\oplus\scrG^\loc{}_{1}$, where
\begin{align}
\scrG^\loc{}_{-1}&= \scrG_{-1} U(\scrG_0)\,, &
\scrG^\loc{}_1&=\scrG_1 U(\scrG_0)\,, &
\scrG^\loc{}_0&=U(\scrG_0)\,.
\end{align}
As noted in \cite{Cederwall:2022oyb}, although $\scrG^\glob$ is not fully associative,
it is sufficiently associative for 
the vector space $\scrG^\loc$ to be a local Lie superalgebra with the commutator $[-,-]$ as the bracket, and as such we denote it by
$\scrG^\Loc$. 
From its subalgebra generated by $\scrG_{-1}\scrG_0$ and $\scrG_1$, we factor out the maximal ideal intersecting $\scrG^\loc{}_0$ trivially
(the maximal \textbf{peripheral} ideal).
We call the resulting local Lie superalgebra the \textbf{local cartanification} of $\scrG^\LLoc$
and the minimal $\mathbb{Z}$-graded
Lie superalgebra with this local part the \textbf{cartanification} of $\scrG^\LLoc$
(or a cartanification of any $\mathbb{Z}$-graded Lie algebra with local part $\scrG^\LLoc$).
We denote the latter by $\scrG^{W}$.

We will in the rest of the paper only consider 
cartanifications of contragredient local Lie superalgebras $\scrB$ which are given by a triple $(\fg,\lambda,\kappa)$
as described in Section~\ref{Contragredient Lie superalgebras}, with a consistent $\mathbb{Z}$-grading.

\begin{ex}\label{W-example}
Let us study the case $(\fg,\lambda,\kappa)=(A_r,\Lambda_1,\kappa_0)$. Set $n=r+1$ so that $A_r=\sl(n)$.
Then $\scrB=A(0,r)=\sl(n|1)$ with a consistent $\mathbb{Z}$-grading $\scrB=\scrB_{-1}\oplus\scrB_{0}\oplus\scrB_{1}$,
where $\scrB_{0}=\gl(n)$ and $\scrB_{\pm1}$ are fundamental $n$-dimensional $\gl(n)$-modules, dual to each other.
We introduce bases of these subspaces as follows,
\begin{align}
\scr B_{-1} &=\langle F^a \rangle\,, & \scr B_0 &=\langle K^a{}_b \rangle\,, & \scr B_1 &=\langle E_a \rangle\,,
\end{align}
with brackets
\begin{align}
\dlb E_a , F^b \drb &=-K^a{}_b+\delta_a{}^b K\,, 
& \dlb K^a{}_b , E_c\drb &=-\delta_c{}^a E_b\,, 
& \dlb K^a{}_b , F^c\drb &=\delta_b{}^c F^a\,,
\end{align}
where $a,b,\ldots=1,2,\ldots,n$ and $K=K^1{}_1 +\cdots+K^n{}_n=-L$.
The bilinear form is given by the Kronecker delta, $\langle E_a | F^b \rangle =\delta_a{}^b$.
In $\scrB^{\,\glob}$, we then get
\begin{align}
F^a E_b = -\dlb F^a, E_b\drb - \langle E_b | F^a \rangle L = K^a{}_b-\delta_a{}^b K+\delta_a{}^b K=K^a{}_b\,
\end{align}
by (\ref{xyplusminus}).
In the cartanification $\scrB^W$, we have to set to zero all linear combinations of elements $F^aK^b{}_c$ in $(\scrB^{\,\glob})_{-1}$
which have zero bracket with all elements in $(\scrB^{\,\glob})_{1}$. Since
\begin{align}
[F^aK^b{}_c, E_d]=F^a[K^b{}_c, E_d]+[F^a, E_d]K^b{}_c=-\de_d{}^bK^a{}_c+\de_d{}^aK^b{}_c\,,
\end{align}
these linear combinations are the ones that are symmetric in the upper indices, and only the antisymmetric ones survive.
It is then straightforward to extend the local Lie superalgebra to the minimal $\mathbb{Z}$-graded Lie superalgebra with this local part,
the cartanification $\scrB^W$,
and find that it is the Lie superalgebra of Cartan type $W(n)$. 
\qed
\end{ex}

\begin{rmk}\label{cliffrmk} Computing the associator $(F^a E_b)F^c-F^a (E_bF^c)$ in $\scrB^\glob$ in the previous example, we get
\begin{align}
(F^a E_b)F^c - F^a (E_bF^c) &= [K^a{}_b,F^c]+F^c(F^a E_b)-
F^a[E_b,F^c]+F^a (F^cE_b)\nn\\
&=F^c(F^a E_b)+F^a (F^cE_b)=(F^cF^a+F^aF^c)E_b\,,
\end{align}
which suggests that we may factor out the ideal generated by $[F^a,F^b]=[E_a,E_b]=0$\linebreak and be left with the associative
Clifford algebra generated by $F^a$ and $E_a$ modulo these relations and $[E_a,F^b]=\delta_a{}^b$.
The Lie superalgebra $W(n)$ is the subalgebra of the commutator algebra spanned by elements of the form $xE_a$, where $x$
belongs to the Grassmann subalgebra generated by $F^a$.
\qed
\end{rmk}

\noindent
By replacing $\scrG_{-1}\scrG_0$
in the definition of the Cartanification above
by $\scrG_{-1}\scrG_0{}'$ for some subalgebra $\scrG_0{}'$ of $\scrG_0$, we get a \textbf{restricted cartanification}.
In particular, with $\scrG_0{}'=\fg^-$
(for $\scrG=\scrB$)
we get a restricted cartanification that we call the \textbf{strong cartanification}
and denote by $\scrB^S$ (the original $\scrB^W$ being the \textbf{weak cartanification}). 

\begin{ex} In the same case as in Example~\ref{W-example} and Remark~\ref{cliffrmk}, we have $\scrB^S=S(n)$, which is the subalgebra of $W(n)$
spanned by elements $xE_a$ such that $[x,E_a]=0$.
\qed
\end{ex}

\subsection{Structure of $(\scrB^W)_{-1}$} \label{sec3.3}

We will in now see that, under certain conditions, the structure of $(\scrB^W)_{-1}$ as a $\fg$-module 
coincides with the one of $W_{-1}$, described in Section~\ref{tha-structure}. This suggests that 
the algebras $\scrB^W$ and $W$ are isomorphic under these conditions, 
and in Section~\ref{isomorphism-section} we will then see that this indeed is the case.

For any weight $\mu=\sum_{i=1}^r \mu_i \Lambda_i$ of $\fg$, set
\begin{align}
\mu^\irt&=\sum_{i=1}^r \mu_i{}^\irt \Lambda_i=
\sum_{i=1}^r {\epsilon_i}\mu_i \Lambda_i\,.
\end{align}
\begin{defi}
A weight $\mu$ of $\fg$ is \textbf{pseudo-minuscule} if it is a dominant integral weight and $(\mu^\irt,\alpha)=0,1$ for any
root $\alpha$ of $\fg$.\qed
\end{defi}

\noindent
The definition implies that $\fg$ is finite-dimensional.
It is slightly more general than the one in \cite{Cederwall:2022oyb},
not only since we now allow $\fg$ to be non-simple (but still semisimple), but also
since we now consider the zero weight as pseudo-minuscule. A nonzero pseudo-minuscule weight $\mu$ can also be characterised by
$(\mu^\irt,\theta)=1$, where $\theta$ is a highest weight of $\fg$. If $\fg$ is simple, this means that
$\mu$ is a fundamental weight such that the corresponding Coxeter label is equal to $1$ (see \cite{Cederwall:2022oyb} for a list). From the point of view
of extended geometry, these cases are characterised by the absense of so-called ancillary transformations \cite{Cederwall:2017fjm}.

Since $(\lambda^\tri)^\irt=\lambda$, the weight $\lambda^\tri$ is pseudo-minuscule if and only if it is a dominant integral weight and
$(\lambda,\beta)=0,1$
for any root $\beta$ of $\fg$.
In the extension to $\scrB$ defined by $\lambda$, this condition is furthermore equivalent to $(\alpha_0{}^\crt,\beta)=0,1$
for any root $\alpha$ of $\fg$.
It then follows that
\begin{align} \label{prop3.3}
\big((\alpha_0{}^\crt\!,\alpha)+1\big)\dlb f_0,e_\alpha \drb =0
\end{align}
for any root $\alpha \neq \alpha_0$ of $\scr B$ with corresponding root vector $e_\alpha \in \scr B_1$ \cite[Proposition~4.1]{Cederwall:2022oyb}.

We will now study the structure of $(\scrB^W)_{-1}$ when $\lambda^\tri$ is a pseudo-minuscule weight.

\begin{prop} \label{sistalemmat}
If $\lambda^\tri$ is pseudo-minuscule, then
\begin{align}
f_0(h_0+L)&=0\,,
& f_0e_\beta&=0\,\label{relationsinhatW3}
\end{align}
in $\scrB^{W}$
for any root $\beta$ of $\fg$ such that 
$(\lambda,\beta)=1$. In particular, $f_0e_j=0$ for any $j\in J$.
\end{prop}
\begin{proof}
Let $e_\alpha$ be a root vector corresponding to a root $\alpha$ of $\scr B$ such that $e_\alpha \in \scr B_1$ but $\alpha \neq \alpha_0$.
We then have
\begin{align}
[e_0,f_0 h_0]+[e_0,f_0 L]&= [e_0,f_0]h_0-f_0[e_0,h_0]+[e_0,f_0]L-f_0[e_0,L]\nn\\
&=h_0+L+f_0e_0=h_0+L-h_0-L=0\,,
\end{align}
and, using (\ref{prop3.3}), also
\begin{align}
[e_\alpha,f_0h_0]+[e_\alpha,f_0L]&=[e_\alpha,f_0]h_0-f_0[e_\alpha,h_0]+[e_\alpha,f_0]L-f_0[e_\alpha,L]\nn\\&=f_0[h_0,e_\alpha]+f_0[L,e_\alpha]
=\big((\alpha_0{}^\crt,\alpha)+1\big)f_0e_\alpha\nn\\&=-\big((\alpha_0{}^\crt,\alpha)+1\big)\dlb f_0,e_\alpha\drb=0\,.
\end{align}
Thus $f_0(h_0+L)$ generates a peripheral ideal of the subalgebra 
of $\scrB^\Loc$ generated by $\scrB_{-1}\scrB_0$ and $\scrB_1$,
and thus
$f_0(h_0+L)=0$ in $\scr B^{W}$. 
The second relation in (\ref{relationsinhatW3}) then follows from the first one by
\begin{align}
f_0e_\beta=-f_0[h_0+L,e_\beta]=-[f_0(h_0+L),e_\beta]=0\,.
\end{align}
\end{proof}
\noindent
When also $\lambda$ is pseudo-minuscule, which means that $\lambda=\lambda^\tri$, we can say more.

\begin{prop} \label{jprop}
If $\lambda^\tri$ is pseudo-minuscule and $\lambda=\lambda^\tri$, then
\begin{align}
f_0f_j-[f_0,f_j]h_j=0 
\end{align}
in $\scrB^{W}$ for any $j\in J$. 
\end{prop}
\begin{proof}
We will show that $[f_0f_j,e_\alpha]-[[f_0,f_j]h_j,e_\alpha]=0$ for any $e_\alpha \in \scr B_1$, divided into five cases:
\begin{enumerate}
\item[(i)] $\alpha=\alpha_0$,
\item[(ii)] $\alpha=\alpha_0+\alpha_j$,
\item[(iii)] $(\lambda,\alpha)=1$ but $\alpha\neq\alpha_0+\alpha_j$, which implies
\begin{align}
[ f_j,e_\alpha ]=
\dlb f_j,f_j,f_0,e_\alpha \drb= 
\dlb e_j,f_0,e_\alpha \drb=0 \label{fall3relation}
\end{align}
\item[(iv)] $(\lambda,\alpha)=2$, which implies
\begin{align}
\dlb f_0,e_\alpha \drb=\dlb f_j,f_0,f_j,e_j,e_\alpha\drb=[ f_j,f_j,e_\alpha]=0\,, \label{fall4relation}
\end{align}
\item[(v)] $(\lambda,\alpha)\geq3$.
\end{enumerate}
The relations (\ref{fall3relation}) and (\ref{fall4relation})
will be needed in a moment and follow easily from the fact that $\lambda^\tri$ is pseudo-minuscule.
In cases (i) and (ii), the proof consists of straightforward calculations: a short one,
\begin{align}
[f_0f_j,e_0]-[[f_0,f_j]h_j,e_0]&=f_j+[f_0,f_j]e_0 \nn\\
&=f_j-\dlb [f_0,f_j],e_0 \drb
=f_j-[h_0,f_j]=0\,,
\end{align}
in case (i) and a more lengthy one in case (ii), which we omit.

In cases (iii)--(v), we have $[f_0,e_\alpha]=[[f_j,f_0],e_\alpha]=0$
and then
\begin{align}
[f_0f_j,e_\alpha]-[[f_0,f_j]h_j,e_\alpha]&=f_0[f_j,e_\alpha]-[f_j,f_0][h_j,e_\alpha]\nn\\
&=-\dlb f_0,[f_j,e_\alpha] \drb + (\alpha_j{}^\crt,\alpha)\dlb [f_j,f_0],e_\alpha \drb\nn\\
&=\big((\alpha_j{}^\crt,\alpha)-1\big) \dlb f_0,[f_j,e_\alpha] \drb+ (\alpha_j{}^\crt,\alpha)[ f_j,\dlb f_0,e_\alpha \drb ]\,.
\end{align}
In case (v) both terms on the last line vanish. In each of the cases (iii) and (iv), one of the terms vanishes, and we have to show that
the other one vanishes too.
In case (iii) we have
\begin{align}
0&=\dlb e_j, f_j,f_j,f_0,e_\alpha \drb \nn\\
&=\dlb h_j,f_j,f_0,e_\alpha \drb + \dlb f_j,h_j,f_0,e_\alpha \drb + \dlb f_j,f_j,e_j,f_0,e_\alpha \drb\nn\\
&=2(\alpha_j{}^\crt,\alpha)\dlb f_j,f_0,e_\alpha \drb\,
\end{align}
and in case (iv) we have
\begin{align}
0&=\dlb e_j,f_0,f_j,f_j,e_\alpha\drb\nn\\
&=\dlb f_0,h_j,f_j,e_\alpha\drb+\dlb f_0,f_j,h_j,e_\alpha\drb+\dlb f_0,f_j,f_j,e_j,e_\alpha\drb\nn\\
&=\big(-2 +2(\alpha_j{}^\crt,\alpha)  \big)\dlb f_0,f_j,e_\alpha\drb+2\dlb f_j,f_0,f_j,e_j,e_\alpha\drb-\dlb f_j,f_j,f_0,e_j,e_\alpha\drb\nn\\
&=2\big((\alpha_j{}^\crt,\alpha)-1   \big)\dlb f_0,f_j,e_\alpha\drb\,,
\end{align}
where we have used the Serre relation $[f_j,f_j,f_0]=0$ which holds since $\lambda_j=\lambda^\tri{}_j=1$.
\end{proof}

\noindent
We are now ready to describe how
$(\scrB^W)_{-1}$ decomposes into $\fg$-modules.

\begin{theorem} \label{minusonestructure}
Let $\lambda^\tri$ be pseudo-minuscule and $\lambda=\lambda^\tri$, and let
$\vartheta$ be a highest root of $\fg^-$. Then the 
submodules $U(\fg)(f_0h_0)$ and $U(\fg)(f_0e_\vartheta)$ of the $\fg$-module $(\scrB^W)_{-1}$
are irreducible with highest weights $\lambda$ and $\vartheta+\lambda$, respectively.
Furthermore, if $\fg$ is simple, then $(\scrB^W)_{-1}$ is the direct sum of all such submodules.
\end{theorem}

\begin{proof}
We first show that $[e_i,f_0e_{\vartheta}]=[e_i,f_0h_0]=0$ for $i\in I$. We have
\begin{align}
[e_i,f_0e_{\vartheta}]&=f_0[e_i,e_{\vartheta}]\,, & [e_i,f_0h_0]&=f_0[e_i,h_0]\,,
\end{align}
where the right hand sides clearly are zero for $i\in K$. When $i\in J$, the right hand sides are zero by Proposition~\ref{sistalemmat}.
Thus $U(\fg)(f_0h_0)$ and $U(\fg)(f_0e_\vartheta)$ are highest-weight modules
with highest weights $\lambda$ and $\vartheta+\lambda$, respectively, and since
they are finite-dimensional, they are irreducible, and the sum of all of them is direct.
Denote this direct sum by $V$, so that
\begin{align}
V=U(\mathfrak{g})(f_0e_{\vartheta}) \oplus U(\mathfrak{g})(f_0h_0)\,
\end{align}
when $\fg^-$ is simple. In the general case, where $\fg^-$ is not simple but semisimple, $U(\mathfrak{g})(f_0e_{\vartheta})$
is replaced by a direct sum of such terms, one for each simple component of $\fg^-$.

It remains to show that when $\fg$ is simple, any element $x_{-1}y_0\in (\scrB^{W})_{-1}$,
where $x_{-1}\in\scr B_{-1}$ and $y_0\in\scr B_0$, can be written as a sum of elements in $V$.
To this end, we decompose $\scr B_0$ into the direct sum of subspaces
\begin{align}
\scr B_0 = \scr B_{(0,-1)} \oplus \scr B_{(0,0)} \oplus \scr B_{(0,1)}
\end{align}
where the subspaces
$\scr B_{(0,\pm1)}$ are spanned by root vectors $e_\alpha$ with $(\lambda,\alpha)=\pm1$ and
\begin{align}
\scr B_{(0,0)} &= \fg^- \oplus \langle h_0 \rangle \oplus \langle L \rangle\,,
\end{align}
since $\fg$ is simple, so there is only one $h_j$ with $j\in J$, which can be written as a linear combination of $L$ and $h_k$ for $k\in K$.
In Proposition~\ref{sistalemmat} we have already shown that
$f_0\scr B_{(0,1)}=0$, and that $f_0L=-f_0h_0$. Furthermore, each simple component of $\fg^-$ equals $U(\fg^-)(e_{\vartheta})$
where $\vartheta$ is the corresponding highest root and 
\begin{align}
f_0\big( U(\mathfrak{g}^-)(e_{\vartheta})\big)
=U(\mathfrak{g}^-)(f_0e_{\vartheta})\subseteq U(\mathfrak{g})(f_0e_{\vartheta})\,.
\end{align}
We have thus shown that
\begin{align}
f_0\scrB_{(0,0)} \oplus f_0\scrB_{(0,1)} \subseteq V\,.
\end{align}
In order to show that $f_0\scrB_{(0,-1)} \subseteq V$, we first consider
$f_0f_j$ for $j\in J$. From Proposition~\ref{jprop} we know that $f_0f_j=[f_0,f_j]h_j$
and it follows that 
\begin{align}
f_0f_j=2f_0f_j-f_0f_j=-f_0[h_j,f_j]-[f_0,f_j]h_j=-[f_0h_j,f_j]\,.
\end{align}
Since $f_0h_j \in f_0 \scrB_{(0,0)}\subseteq V$, we get $f_0f_j\in V$. Now
\begin{align}
f_0 \scr B_{(0,-1)}=f_0 U(\mathfrak{g}^-)(f_j)=U(\mathfrak{g}^-)(f_0f_j)\subseteq U(\mathfrak{g})(f_0f_j) \subseteq V
\end{align}
and we have shown that $f_0 \scr B_{0} \subseteq V$. 
Since
\begin{align}
\scr B_{-1} \scr B_{0} = \big(U(\mathfrak{g})(f_0)\big)\scr B_{0}\subseteq U(\mathfrak{g})(f_0\scr B_{0})\,,
\end{align}
this then generalises to $\scr B_{-1} \scr B_{0} \subseteq V$.
\end{proof}

\begin{ex} Consider $(\fg,\lambda,\kappa)=(A_4,\Lambda_2,\kappa_0)$. With $K^a{}_b$ and $K$ defined as in Example~\ref{W-example} with $n=5$,
we have
\begin{align}
\dlb E_{ab}, F^{cd} \drb &= -4\de_{[a}{}^{[c} K^{d]}{}_{b]} + 2\de_{[a}{}^c{} \de_{b]}{}^d{} K\,, &
[K^a{}_b,E_{cd}]&=2\de_{[c}{}^a{}E_{d]b}\,, \nn\\
\langle E_{ab}|F^{cd} \rangle&= 2\de_{[a}{}^c{} \de_{b]}{}^d{}\,,  \qquad \quad \  L=-\frac12K\,,  & [K^a{}_b,F^{cd}]&=-2\de_{b}{}^{[c}{}F^{d]a}\,.
\end{align}
(Square brackets denote antisymmetrisation, for example
$2\de_{[c}{}^a{}E_{d]b}=\de_{c}{}^a{}E_{db}-\de_{d}{}^a{}E_{cb}$, and below round brackets denote symmetrisation in the same way.)
By removing node $2$ in the Dynkin diagram of $\fg=A_4$, we get the Dynkin diagram of the non-simple subalgebra $\fg^-=A_1\oplus A_2$.
The highest root of the 
$A_1$ component is $\alpha_1=2\Lambda_1-\Lambda_2$, and the highest root of the 
$A_2$ component is $\alpha_3+\alpha_4=-\Lambda_2+\Lambda_3+\Lambda_4$. According to Theorem~\ref{minusonestructure}, 
the $A_4$-module
$\scrB^W$
then decomposes into a direct sum of three irreducible ones, 
one with highest weight $\lambda=\Lambda_2$ and two with 
highest weights 
obtained by adding $\lambda=\Lambda_2$ to these highest roots of $\fg^-$,
which gives $2\Lambda_1$ and
$\Lambda_3+\Lambda_4$. They are spanned by the linear combinations
\begin{align}
\Lambda_2 &: F^{ab}K\,,\nn\\
2\Lambda_1 &: F^{c(a}K^{b)}{}_c\,,\nn\\
\Lambda_3+\Lambda_4 &: F^{abc}{}_{d}=F^{[ab}K^{c]}{}_{d}-\tfrac13 F^{[ab}K\delta_{d}{}^{c]}-\tfrac23 F^{e[a}K^{b}{}_{e}\delta_{d}{}^{c]}\,,
\end{align}
with summation on repeated indices. A straightforward calculation gives
\begin{align}
[E_{ab},F^{cde}{}_f]=6\de_{a}{}^{[c}\de_{b}{}^{d}K^{e]}{}_f+4\de_{f}{}^{[c}\de_{[a}{}^{d}K^{e]}{}_{b]}-2
\de_{a}{}^{[c}\de_{b}{}^{d}\de_{f}{}^{e]} K\,
\end{align}
and setting $3F^{abc}{}_{d}=\varepsilon^{abcef}F_{ef|d}$, this is equivalent to
\begin{align}
[E_{ab},F_{cd|e}]=\varepsilon_{abcdf}K^f{}_e + \varepsilon_{abcef}K^f{}_d
\end{align}
(where $\varepsilon$ is fully antisymmetric and $\varepsilon^{12345}=\varepsilon_{12345}=1$)
showing that the restricted local cartanification of $\scrB$ with respect to the $A_2$ subalgebra is the local part of
the exceptional linearly compact Lie superalgebra $E(5,10)$ \cite{Kac98,ChengKac,Shchepochkina}. Applying the isomorphism in the next section,
$E(5,10)$ can thus be considered as a ``restricted tensor hierarchy algebra'', 
as was noted in \cite{Cederwall:2021ejp},
with generators and relations
obtained by removing not only $h_0$ and $f_{00}$ but also $f_{01}$
from the ones for the associated ``unrestricted'' tensor hierarchy algebra $W$.
Similarly, the
restricted local cartanification of $\scrB$ with respect to the $A_2$ subalgebra is the local part of
the ``strange'' Lie superalgebra $P(5)$, and this construction can be generalised to $P(n)$ for any $n$.
\qed
\end{ex}

\section{The isomorphism between $W$ and $\scr B^W$} \label{isomorphism-section}

We will now show that when the triple $(\fg,\lambda,\kappa)$ is such that $\fg$ is simple, $\lambda$ is a pseudo-minuscule
weight and $\lambda=\lambda^\tri$, then $\scr B^W$ constructed in the preceding section is isomorphic to
the tensor hierarchy algebra $W$ constructed from $(\fg,\lambda,\kappa)$ by generators and relations in Section~\ref{tha-subsection}.
We begin by showing that there is a homomorphism when $\lambda^\tri$ is pseudo-minuscule, and then that this homomorphism is bijective in the
case where $\fg$ is simple and $\lambda=\lambda^\tri$.

\begin{prop} \label{satsen}
Let $\phi$ be the linear map
from the local part of $W$ to the local part of $\scrB^W$
given by
$f_{0i}\mapsto f_0 h_i$ for $i \in \{0\} \cup K$, leaving the other generators unchanged. Then $\phi$ is a local Lie superalgebra homomorphism.
\end{prop}

\begin{proof}
We will show that the relations (\ref{ordchevrel})--(\ref{addrel3}) are satisfied in $\scrB^{W}$ with $f_{0i}=f_0 h_i$.
We will first
show that $[f_0h,e_j]=[f_j,[f_j,f_0h]]=0$ for any $h \in \mathfrak{h}$ and $j\in J$.
From Proposition~\ref{sistalemmat} we know that $f_0(h_0+L)=f_0e_j=0$ in $\scrB^W$.
We get
\begin{align}
[f_0h,e_j]&=f_0[h,e_j]=\alpha_j(h)f_0e_j=0\,
\end{align}
for any $h\in\mathfrak{h}$. Similarly, since $({\rm ad}\,f_j)$ acts on $f_{0i}=f_0h_i$ by the Leibniz rule,
and since
\begin{align}
({\rm ad}\,f_j)^{1-B_{j0}}(f_{0})=({\rm ad}\,f_j)^{2}(h_i)=({\rm ad}\,f_j)^{2}(h_0+L)=0\,,
\end{align}
we have
\begin{align}
({\rm ad}\,f_j)^{1-B_{j0}}(f_{0}h_i)&=({1-B_{j0}})({\rm ad}\,f_j)^{-B_{j0}}(f_{0})[f_j,h_i]\nn\\
&=B_{ij}({1-B_{j0}})({\rm ad}\,f_j)^{-B_{j0}}(f_{0})f_j\nn\\
&=\frac{B_{ij}}{B_{0j}}({1-B_{j0}})({\rm ad}\,f_j)^{-B_{j0}}(f_{0})[f_j,h_0+L]\nn\\
&=\frac{B_{ij}}{B_{0j}}({\rm ad}\,f_j)^{1-B_{j0}}\big(f_0(h_0+L)\big)=0
\end{align}
where the last equality follows from Proposition~\ref{sistalemmat}.
We have thus shown the relations (\ref{addrel2}).
The relations (\ref{eifjf0a}) and (\ref{addrel1})  
are straightforward to show, 
\begin{align}
[e_0,f_{0i}]&=[e_0,f_{0}h_{i}]=[e_0,f_0]h_i-f_0[e_0,h_i]=h_i\,,\nn\\
[h_i,f_{0j}]&=[h_i,f_{0}h_{j}]=[h_i,f_{0}]h_j+f_0[h_i,h_j]=-B_{i0}f_0h_j=-B_{i0}f_{0j}\,,\nn\\
[e_k,[f_l,f_{0i}]]&=[e_k,[f_l,f_{0}h_{i}]]
=f_0[e_k,[f_l,h_i]]=\de_{kl}B_{il}f_0h_k=\de_{kl}B_{il}f_{0k}\,.
\end{align}
The relations (\ref{addrel4}) and (\ref{addrel3}) easily follow from $[e_k,[e_k,h_i]]=[f_k,[f_k,h_i]]=0$ and $f_0[e_j,[e_k,h_i]]=0$, respectively,
where the last equality in turn follows from $f_0e_\beta=0$ in Proposition~\ref{sistalemmat}.
The relations not involving $f_{0i}$ are automatically satisfied.
\end{proof}

\begin{theorem} \label{maintheorem}
If $\fg$ is simple, $\lambda$ is a pseudo-minuscule
weight and $\lambda=\lambda^\tri$, then $W$ and $\scrB^W$ are isomorphic.
\end{theorem}
\begin{proof}
We will show that the local Lie superalgebra homomorphism $\phi$ is bijective.

The surjectivity of $\phi$ follows from Theorem~\ref{minusonestructure} together with the facts that $f_0h_0=\phi(f_{00})$ and
\begin{align}
f_0e_{\vartheta}=\frac1{(\mu^\crt,\vartheta)}f_0[h_\mu,e_{\vartheta}]=
\frac1{(\mu^\crt,\vartheta)}[f_0h_\mu,e_{\vartheta}]=
\frac1{(\mu^\crt,\vartheta)}[\phi(f_{0\mu}),e_{\vartheta}]=-\phi(e_{\vartheta}{}^\sh)\,,
\end{align}
for some weight $\mu$ of $\fg^-$ such that $(\mu^\crt,\vartheta)\neq0$ (for example $\mu=\vartheta$).

For the injectivity, we have just seen that $f_0e_{\vartheta}=-\phi(e_{\vartheta}{}^\sh)$. We can then define a linear map
$\sigma:\scrB^W \to W$ recursively by $\sigma(f_0e_{\vartheta})=e_{\vartheta}{}^\sh$ and
\begin{align}
\sigma[x,y]=[x,\sigma(y)]
\end{align}
for any $x \in \mathfrak{g}$ and it follows by induction that $\sigma(\phi(x))=x$ for any $x \in W$.
However, we need to ensure that $\sigma$ is well defined,
and for this 
it suffices to show that 
$J$ must be such that $(\alpha_j{}^\crt,\alpha)\geq-1$
for any $j\in J$, so that
the conditions in Theorem~\ref{tha-struktur}
is satisfied.
Since $\lambda^\tri$ is a pseudo-minuscule weight, we have $[e_j,[e_j,e_\alpha]]=0$ (otherwise $\alpha+2\alpha_j$ would be a
root with $(\alpha_0{}^\crt,\alpha+2\alpha_j)=-2$). Then
\begin{align}
0&=[f_j,f_j,e_j,e_j,e_\alpha]=2\big(1+(\alpha_j{}^\crt,\alpha)\big)(\alpha_j{}^\crt,\alpha)e_\alpha 
\end{align}
so that indeed either $(\alpha_j{}^\crt,\alpha)=0$ or $(\alpha_j{}^\crt,\alpha)=-1$.

Finally, the isomorphism between the local parts induces an isomorphism between the
corresponding minimal $\mathbb{Z}$-graded Lie superalgebras \cite{Kac68,Kac77B}.
\end{proof}

\section{Conclusions}\label{sec5}

We have seen that the map $\phi$ is an isomorphism 
when $\fg$ is simple and finite-dimensional, $\lambda$ is a pseudo-minuscule
weight and $\lambda=\lambda^\tri$.

What happens when $\fg$ is infinite-dimensional? The procedure 
of the cartanification still makes sense and gives a non-trivial algebra, but it seems quite different from
the corresponding tensor hierarchy algebra $W$.
Consider $(\fg,\lambda,\kappa)=(E_9,\Lambda_1,\kappa_0)$
as an example. Then the Cartan matrix $A$ is not invertible, but we can still extend it to a Cartan matrix
$B$ of $\scrB$, from which we can in turn construct a Lie superalgebra $W$ following the same steps  
as when $A$ is invertible, described in the beginning of Section~\ref{tha-subsection}.
At degree $0$ it contains an element
\begin{align}
{\sf J}=[e_0,e_1,e_\vartheta,f_{02}]
\end{align}
where $\vartheta$ is the highest root of the $E_8$ subalgebra.
As shown in \cite[Section 5]{Cederwall:2021xqi}, we have
$[{\sf J},e_1]\in E_9{}'$ and $[{\sf J},f_1]\in E_9{}'$,
but ${\sf J}\notin E_9{}'$ (where $E_9{}'$ is the derived algebra of $E_9$,
in other words,
the centrally extended loop algebra of $E_8$). More precisely, ${\sf J}$ acts as the Virasoro generator $L_1$,
which is not part of $E_9{}'$ but can be added to $E_9{}'$ equally well as the derivation identified with $L_0$
in the standard definition of the affine Kac--Moody algebra $E_9$. However, in the cartanification, the corresponding element
\begin{align}
[e_0,e_1,e_\vartheta,f_0h_2]
\end{align}
is zero. Indeed, $[e_\vartheta,f_0h_2]=f_0[e_\vartheta,h_2]=-f_0e_\vartheta$ and
\begin{align}
[e_0,e_1,f_0e_\vartheta]&=[e_0,f_0[e_1,e_\vartheta]]\nn\\
&=[e_0,f_0][e_1,e_\vartheta]-f_0[e_0,e_1,e_\vartheta]\nn\\
&=[e_1,e_\vartheta]+\dlb f_0,[e_0,e_1,e_\vartheta]\drb\nn\\
&=[e_1,e_\vartheta]+[h_0,e_1,e_\vartheta]=0\,. \label{e9calc1}
\end{align}

How can the cartanification be generalised in order to cover the infinite-dimensional cases? Already the proof of
Proposition~\ref{sistalemmat} fails, and in particular we cannot conclude that $f_0e_j=0$ for any $j\in J$.
This is essential since we always have $[f_{0k},e_j]=0$ in $W$, and $[f_0h_k,e_j]=B_{kj}f_0e_j$.
Even if we manage to get the relation $f_0e_j=0$ (which means $f_0e_1=0$ in the $E_9$ example above) in some other way,
the problem of the vanishing element
$[e_0,e_1,e_\vartheta,f_0h_2]$ in the $E_9$ example remains. In some sense, it even aggravates, since then already
\begin{align}
[e_1,e_\vartheta,f_0h_2]&=e_1(f_0e_\vartheta)-(f_0e_\vartheta) e_1\nn\\
&=(e_1f_0)e_\vartheta-(f_0e_1) e_\vartheta=0\,, \label{e9calc2}
\end{align}
implying that the whole highest weight module at degree $-1$ with this highest weight vector vanishes, and the attempt to reproduce
the tensor hierarchy algebra $W$ would only result in the Borcherds--Kac--Moody superalgebra $\scr B$. The solution is probably
that associativity needs to be restricted already at this level, so that the naive calculations (\ref{e9calc1}) and (\ref{e9calc2}) do not hold.

In order to overcome the obstacles with infinite-dimensional $\fg$,
it is probably a good strategy to first study the finite-dimensional cases where the conditions are not satisfied,
which in the series $(\fg,\lambda,\kappa)=(E_r,\Lambda_1,\kappa_0)$
means $r=8$. Another way of obtaining the corresponding tensor hierarchy algebras
in this (and similar) series by some kind of modified cartanification would be to consider the non-consistent $\mathbb{Z}$-grading
associated to node $d$, where $r-6\leq d\leq r-3$. The local part would then be a restricted cartanification of the 
contragredient local Lie superalgebra associated to $(\fg,\lambda,\kappa)=(E_{d+1},\Lambda_1,\kappa_0)$
tensored with the Clifford algebra on $2d$ generators $F^a$ and $E_a$ where $a=1,2,\ldots,d$ \cite{CederwallPalmkvistGravityLine}.
For more general decompositions, the Clifford algebra could possibly be replaced by some non-associative generalisation. As
noted in Remark~\ref{cliffrmk}, it seems that the Clifford algebra can be obtained by factoring out the ideal of 
$\scrB^\glob$ generated by $f_0f_0$ and $e_0e_0$ in the case $(\fg,\lambda,\kappa)=(A_{d-1},\Lambda_1,\kappa_0)$.
Do $f_0f_0$ and $e_0e_0$ generate a proper ideal also in the general case?

Finally, a natural direction for further research would be to compare the cartanification not only 
to the construction by generators and relations, referred to as (iii) in the introduction (Section~\ref{intro}),
but also to (i) and (ii). 
In some cases where $\lambda^\tri$ is pseudo-minuscule
but $\lambda$ is not, preliminary studies indicate a
disagreement already between (i) and (iii),
where the cartanification agrees with (i) but not with (iii).
It would be interesting to see whether this holds generally in such cases.

\bibliographystyle{utphysmod2}

\providecommand{\href}[2]{#2}



\end{document}